\documentclass[12pt]{article}
\usepackage{amssymb,amsmath,amsopn,amsfonts}
\usepackage{latexsym}
\usepackage{array}
\usepackage{graphicx}
\usepackage{amsthm}
\usepackage{colordvi,multicol}
\usepackage{epsf}
\usepackage{tikz}
\usepackage[T1]{fontenc}
\usepackage[utf8]{inputenc}
\usepackage{authblk}
\usepackage{hyperref}
\newtheorem{lemma}{Lemma}[section]
\newtheorem{Thm}[lemma]{Theorem}
\newtheorem{remark}[lemma]{Remark}
\newtheorem{definition}[lemma]{Definition}
\newtheorem{hyp}[lemma]{Hypothesis}
\newtheorem{prop}[lemma]{Proposition}

\newfam\msbfam
\font\tenmsb=msbm10 \textfont\msbfam=\tenmsb \font\sevenmsb=msbm7
\scriptfont\msbfam=\sevenmsb \font\fivemsb=msbm5
\scriptscriptfont\msbfam=\fivemsb

\providecommand{\keywords}[1]{\textbf{\textbf{Keywords:}} #1}

\begin{document}
\bigbreak

\title{\bf Moderate deviation principle for the two-dimensional stochastic Navier-Stokes equations with anisotropic viscosity \thanks{Research supported by the DFG through the CRC 1283 "Taming uncertainty and profiting from randomness and low regularity in analysis, stochastics and their applications".}}
\author[1,2]{\bf Bingguang Chen\thanks{bchen@math.uni-bielefeld.de}}
\date{}
\affil[1]{Academy of Mathematics and System Science, Chinese Academy of Science, Beijing 100190, China }
\affil[2]{Department of Mathematics, University of Bielefeld, D-33615 Bielefeld, Germany}
\renewcommand*{\Affilfont}{\small\it}
\renewcommand\Authand{, }

\maketitle

\begin{abstract}
In this paper, we prove a central limit theorem and establish a moderate deviation principle for  the the two-dimensional stochastic Navier-Stokes equations with anisotropic viscosity. The proof for moderate deviation principle is based on the weak convergence approach. 
\end{abstract}

\keywords{Stochastic Navier-Stokes equations; Anisotropic viscosity; Central limit theorem; Moderate deviation principle;    Weak convergence approach }
\vspace{2mm}

\section{Introduction}

The main aim of this work is to establish central limit theorem and moderate deviation principle  for the stochastic Navier-Stokes equation with anisotropic viscosity. We consider the following stochastic Navier-Stokes equation with anisotropic viscosity on the two dimensional (2D) torus $\mathbb{T}^2=\mathbb{R}^2/(2\pi\mathbb{Z})^2$:
\begin{equation}\label{equation before projection}\aligned
&du=\partial^2_1 udt-u\cdot \nabla u dt+\sigma(t, u)dW(t)-\nabla pdt,\\
&\text{div } u=0,\\
&u(0)=u_0,
\endaligned
\end{equation}
where $u(t,x)$ denotes the velocity field at time $t\in[0,T]$ and position $x\in\mathbb{T}^2$, $p$ denotes the pressure field, $\sigma$ is the random external force and $W$ is an $l^2$-cylindrical Wiener process. 

 Let's first recall the classical Navier-Stokes (N-S) equation which is given by
\begin{equation}\label{equation classical}\aligned
&du=\nu\Delta udt-u\cdot \nabla u dt-\nabla pdt,\\
&\text{div } u=0,\\
&u(0)=u_0,
\endaligned
\end{equation}
where $\nu>0$ is the viscosity of the fluid. (\ref{equation classical}) describes the time evolution of an incompressible fluid. In 1934, J. Leray proved global existence of finite energy weak solutions for the deterministic case in the whole space $\mathbb{R}^d$ for $d=2, 3$ in the seminar paper \cite{L33}. For more results on deterministic N-S equation, we refer to \cite{CKN86}, \cite{Te79}, \cite{Te95}, \cite{KT01} and reference therein. For the stochastic case, there exists a great amount of literature too. The existence and uniqueness of solutions and ergodicity property to the stochatic 2D Navier-Stokes equation have been obtained (see e.g. \cite{FG95}, 
\cite{MR05}, \cite{HM06}). Large deviation principles for the two-dimensional stochastic N-S equations  have been established in \cite{CM10} and \cite{SS06}. Moderate deviation principles for the two-dimensional stochastic N-S equations  have been established in \cite{WZZ}.

Compared to \eqref{equation classical}, \eqref{equation before projection} only has partial dissipation, which can be viewed as an intermediate equation between N-S equation and Euler equation.  System of this type appear in in geophysical fluids (see for instance \cite{CDGG06} and \cite{Ped79}). Instead of putting the classical viscosity $-\nu\Delta$ in (\ref{equation classical}), meteorologist often modelize turbulent diffusion by putting a viscosity of the form: $-\nu_h\Delta_h-\nu_3\partial^2_{x_3}$, where $\nu_h$ and $\nu_3$ are empiric constants, and $\nu_3$ is usually much smaller than $\nu_h$. We refer to the book of J. Pedlovsky \cite[Chapter 4]{Ped79} for a more complete discussion. For the 3 dimensional case there is no result concerning global existence of weak solutions. 

In the 2D case, \cite{LZZ18} investigates both the deterministic system  and the stochastic system (\ref{equation before projection}) for $H^{0,1}$ initial value (For the definition of space see Section 2). The main difference in obtaining the global well-posedness for \eqref{equation before projection} is that the $L^2$-norm estimate is not enough to establish $L^2([0,T], L^2)$ strong convergence due to lack of compactness in the second direction. In \cite{LZZ18}, the proof is based on an additional $H^{0,1}$-norm estimate.  In this paper, we want to investigate deviations of stochastic Navier-Stokes equations from the deterministic case.

The large deviation theory concerns the asymptotic behavior of a family of random variables $X_\varepsilon$ and we refer to the monographs \cite{DZ92} and \cite{Str84} for many historical remarks and extensive references. It asserts that for some tail or extreme event $A$, $P(X_\varepsilon \in A)$ converges to zero exponentially fast as $\varepsilon\rightarrow 0$ and the exact rate of convergence is given by the so-called rate function. The large deviation principle was first established by Varadhan in \cite{V66} and he also studied the small time asymptotics of finite dimensional diffusion processes in \cite{Var67}. Since then, many important results concerning the large deviation principle have been established. For results on the large deviation principle for stochastic differential equations in finite dimensional case we refer to \cite{FW84}. For the extensions to infinite dimensional diffusions or SPDE, we refer the readers to \cite{BDM08},  \cite{CM10},  \cite{DM09}, \cite{Liu09}, \cite{LRZ13}, \cite{RZ08}, \cite{XZ08}, \cite{Zh00} and the references therein. 

Moderate deviation is the theory filling in the gap between the central limit theorem and the large deviation principle (see Section 2).  Moderate deviation estimates arise in the theory of statistical inference. It can provide us with the rate of convergence and a useful method for constructing asymptotic confidence intervals, see \cite{E12}, \cite{GZ11}, \cite{IK03}, \cite{K83} and references therein. For the study of MDP for general Markov process see \cite{WuLiming95}.   Resulst of MDP for stochastic partial differential equations have been obtained in \cite{WZ14MDP}, \cite{BDG16MDP},  \cite{DXZZ17MDP} and references therein. 

For $\varepsilon>0$, consider the equation:
\begin{equation}\label{equation after projection with varepsilon}\aligned
&du^\varepsilon(t)=\partial^2_1 u^\varepsilon(t)dt-B(u^\varepsilon(t))dt+\sqrt{\varepsilon}\sigma(t, u^\varepsilon(t))dW(t),\\
&u^\varepsilon(0)=u_0,
\endaligned
\end{equation}
where the definition of $B$ will be given in Section 2.

As $\varepsilon\rightarrow 0$, $u^\varepsilon$  will converges to the solution to the following deterministic equation:
\begin{equation}\label{determinstic equation after projection}\aligned
&du^0(t)=\partial^2_1 u^0(t)dt-B(u^0(t))dt,\\
&u^0(0)=u_0.
\endaligned
\end{equation}

We will investigate deviations of $u^\varepsilon$ from the deterministic solution $u^0$. That is, the asymptotic behaviour of the trajectory
$$\frac{1}{\sqrt{\varepsilon}\lambda(\varepsilon)}(u^\varepsilon-u^0),$$
where $\lambda(\varepsilon)$ is some deviation scale which strongly influence the behaviour.

(1) The case $\lambda(\varepsilon)=\frac{1}{\sqrt{\varepsilon}}$ provides large deviation principle (LDP) estimates, which has been studied in \cite{CZ20}.

(2) If $\lambda(\varepsilon)=1$, we are in the domain of the central limit theorem (CLT). For the study of the central limit theorem for stochastic (partial) differential equation, we refer the readers to \cite{WZZ},  \cite{CLWY18} and \cite{WZ14MDP}.  We will show that $\frac{u^\varepsilon-u^0}{\sqrt{\varepsilon}}$ converges to a solution of a stochastic equation as $\varepsilon\rightarrow 0$ in Section 3.

(3) To fill in the gap between the CLT and LDP, we will study the so-called moderate deviation principle (MDP). The moderate deviation principle refines the estimates obtained through the central limit theorem.  It provides the asymptotic behaviour for $P(\|u^\varepsilon-u^0\|\geqslant \delta \sqrt{\varepsilon}\lambda(\varepsilon))$ while CLT gives bounds for $P(\|u^\varepsilon-u^0\|\geqslant \delta \sqrt{\varepsilon})$.  Throughout this paper we may  assume 
$$\lambda(\varepsilon)\rightarrow\infty, \text{ }\text{ }\text{ }\sqrt{\varepsilon}\lambda(\varepsilon)\rightarrow 0\text{ as }\varepsilon\rightarrow0.$$

We  study  the moderate deviations by using the weak convergence approach. This approach is mainly based on a variational representation formula for certain functionals of infinite dimensional Brownian Motion, which was established by Budhiraja and Dupuis in \cite{BD00}. The main advantage of the weak convergence approach is that one can avoid some exponential probability estimates, which might be very difficult to derive for many infinite dimensional models. To use the weak convergence approach, we need to prove two conditions in Hypothesis \ref{Hyp}. We will  use the argument  in \cite{WZZ}, in which the authors first establish the convergence in  $L^2([0,T],L^2)$  and then by using this and It\^o's formula to obtain $L^\infty([0,T], L^2)\bigcap L^2([0,T], H^{1})$ convergence. As mentioned above, due to  the lack of compactness in the second direction, we need to do $H^{0,1}$ estimate for the skeleton equation \eqref{skeleton eq. MDP}, which requires $H^{0,2}$ estimates of solution to the deterministic equation \eqref{determinstic equation after projection}. To obtain this, we use a commutator estimate (see Lemma \ref{commutator estimates}) from \cite{CDGG00}. This also leads to $H^{0,2}$ condition for the initial value.

{\textbf{Organization of the paper}}

In Section 2, we introduce the basic  notation, definition and recall some preliminary results.  In Section 3, we will build the central limit theorem. In Section 4, we prove the moderate deviation principle  for the the two-dimensional stochastic Navier-Stokes equations with anisotropic viscosity.

\section{Preliminary}

 \vskip.10in
{\textbf{Function spaces on $\mathbb{T}^2$}}

We first recall some definitions of function spaces for the two dimensional torus $\mathbb{T}^2$. 

Let $\mathbb{T}^2=\mathbb{R}/2\pi\mathbb{Z}\times\mathbb{R}/2\pi\mathbb{Z}=(\mathbb{T}_h, \mathbb{T}_v)$ where $h$ stands for the horizonal variable $x_1$ and $v$ stands for the vertical variable $x_2$. For exponents $p,q\in [1, \infty)$, we denote the space $L^p(\mathbb{T}_h, L^q(\mathbb{T}_v))$ by $L^p_h(L^q_v)$, which is endowed with the norm
$$\|u\|_{L^p_h(L^q_v)(\mathbb{T}^2)}:=\{\int_{\mathbb{T}_h}(\int_{\mathbb{T}_v}|u(x_1, x_2)|^qdx_2)^{\frac{p}{q}}dx_1\}^{\frac{1}{p}}.$$

Similar notation for $L^p_v(L^q_h)$. In the case $p, q=\infty$, we denote $L^\infty$ the essential supremum norm. 
Throughout the paper, we denote various positive constants by the same letter $C$.

For $u\in L^2(\mathbb{T}^2)$, we consider the Fourier expansion of $u$:
$$u(x)=\sum_{k\in\mathbb{Z}^2}\hat{u}_ke^{ik\cdot x} \text{ } \text{with} \text{ } \hat{u}_k=\overline{\hat{u}_{-k}},$$
where $\hat{u}_k:=\frac{1}{(2\pi)^2}\int_{[0, 2\pi]\times[0, 2\pi]}u(x)e^{-ik\cdot x}dx$ denotes the Fourier coefficient of $u$ on $\mathbb{T}^2$.

Define the Sobolev norm:
$$\|u\|_{H^s}^2:=\sum_{k\in\mathbb{Z}^2}(1+|k|^2)^s |\hat{u}_k|^2,$$
and the anisotropic Sobolev norm:
$$\|u\|^2_{H^{s, s'}}=\sum_{k\in\mathbb{Z}^2}(1+|k_1|^2)^s(1+|k_2|^2)^{s'}|\hat{u}_k|^2,$$
where $k=(k_1, k_2)$. We define the Sobolev spaces $H^s(\mathbb{T}^2)$, $H^{s,s'}(\mathbb{T}^2)$ as the completion of $C^\infty(\mathbb{T}^2)$ with the norms $\|\cdot\|_{H^s}$, $\|\cdot\|_{H^{s,s'}}$ respectively.  The notation $L^p_v(H^s_h)$ is given by 
$$\|u\|_{L^p_v(H^s_h)}:=\left(\int_{\mathbb{T}_v}\|u(\cdot, x_2)\|^p_{H^s(\mathbb{T}_h)}dx_2\right)^\frac{1}{p}$$

Let us recall the definition of anisotropic dyadic decomposition of the Fourier space, which will lead to another represnetation of $H^{s,s'}$ in the sense of Besov space. For a general introduction to the theory of Besov space we refer to \cite{BCD11}, \cite{Tri78}, \cite{Tri06}. 

 Let $\chi,\theta\in \mathcal{D}$ be nonnegative radial functions on $\mathbb{R}$, such that

i. the support of $\chi$ is contained in a ball and the support of $\theta$ is contained in an annulus;

ii. $\chi(z)+\sum_{j\geq0}\theta(2^{-j}z)=1$ for all $z\in \mathbb{R}$.

iii. $\textrm{supp}(\chi)\cap \textrm{supp}(\theta(2^{-j}\cdot))=\emptyset$ for $j\geq1$ and $\textrm{supp}\theta(2^{-i}\cdot)\cap \textrm{supp}\theta(2^{-j}\cdot)=\emptyset$ for $|i-j|>1$.

We call such $(\chi,\theta)$ dyadic partition of unity.  The Littlewood-Paley blocks in the vertical variable are now defined as $u=\sum_{j\geqslant -1}\Delta^v_j u$, where
$$\Delta^v_{-1}u=\mathcal{F}^{-1}(\chi(|k_2|)\hat{u})\quad \Delta^v_{j}u=\mathcal{F}^{-1}(\theta(2^{-j}|k_2|)\hat{u}), \text{ }k_2\in \mathbb{Z},$$
where $\mathcal{F}^{-1}$ is the inverse Fourier transform. The anisotropic Sobolev norm can also be defined as follows:
$$\|u\|_{H^{s,s'}}=\left(\sum_{j\geq -1} 2^{2js'}\|\Delta^v_j u\|^2_{L^2_v(H^s(\mathbb{T}_h))}\right)^\frac{1}{2}.$$

To formulate the stochastic Navier-Stokes equations with anisotropic viscosity, we need the following spaces:
$$H:=\{u\in L^2(\mathbb{T}^2; \mathbb{R}^2); \text{div}\text{ } u=0\},$$
$$V:=\{u\in H^1(\mathbb{T}^2; \mathbb{R}^2); \text{div}\text{ } u=0\},$$
$$\tilde{H}^{s,s'}:=\{u\in H^{s,s'}(\mathbb{T}^2;\mathbb{R}^2); \text{div}\text{ } u=0\}.$$
Moreover, we use $\langle\cdot, \cdot\rangle $ to denote the scalar product (which is also the inner product of $L^2$ and $H$)
$$\langle u, v\rangle =\sum_{j=1}^2\int_{\mathbb{T}^2}u^j(x)v^j(x)dx$$
and $\langle\cdot, \cdot\rangle_X$ to denote the inner product of Hilbert space $X$ where $X=l^2$, $V$ or $\tilde{H}^{s,s'}$.

Due to the divergence free condition, we need the Larey projection operator $P_H: \text{ } L^2(\mathbb{T}^2)\rightarrow H$:
$$P_H: u\mapsto u-\nabla \Delta^{-1}(\text{div } u).$$
 
By applying the operator $P_H$ to (\ref{equation before projection})
we can rewrite the equation in the following form:
\begin{equation}\label{equation after projection}\aligned
&du(t)=\partial^2_1 u(t)dt-B(u(t))dt+\sigma(t, u(t))dW(t),\\
&u(0)=u_0,
\endaligned
\end{equation}
where the nonlinear operator $B(u,v)=P_H(u\cdot \nabla v)$ with the notation $B(u)=B(u,u)$. Here we use the same symbol  $\sigma$ after projection for simplicity.

For $u,v,w\in V$, define
$$b(u,v,w):=\langle B(u,v), w\rangle.$$
We have $b(u,v,w)=-b(u,w,v)$ and $b(u,v,v)=0$.

 We put some estimates of $b$ in the Appendix.  


 \vskip.10in
{\textbf{Large deviation principle}}

We recall the definition of the large deviation principle. For a general introduction to the theory we refer to \cite{DZ92}, \cite{DZ98}.

\begin{definition}[Large deviation principle]
Given a family of probability measures $\{\mu_\varepsilon\}_{\varepsilon>0}$ on a  metric space $(E,\rho)$ and a lower semicontinuous function $I:E\rightarrow [0,\infty]$ not identically equal to $+\infty$. The family $\{\mu_\varepsilon\}$ is said to satisfy the large deviation principle(LDP) with respect to the rate function $I$ if\\
(U) for all closed sets $F\subset E$ we have 
$$\limsup_{\varepsilon\rightarrow 0}\varepsilon\log\mu_\varepsilon(F)\leqslant-\inf_{x\in F}I(x),$$
(L) for all open sets $G\subset E$ we have 
$$\liminf_{\varepsilon\rightarrow 0}\varepsilon\log\mu_\varepsilon(G)\geqslant-\inf_{x\in G}I(x).$$

A family of random variable is said to satisfy large deviation principle if the law of these random variables satisfy large deviation princple.

Moreover, $I$ is a good rate function if its level sets $I_r:=\{x\in E:I(x)\leqslant r\}$ are compact for arbitrary $r\in (0,+\infty)$.
\end{definition}
 \vskip.10in

\begin{definition}[Laplace principle]
A sequence of random variables $\{X^\varepsilon\}$ is said to satisfy the Laplace principle with rate function $I$ if for each bounded continuous real-valued function $h$ defined on $E$
$$\lim_{\varepsilon\rightarrow 0}\varepsilon\log E\left[e^{-\frac{1}{\varepsilon}h(X^\varepsilon)}\right]=-\inf_{x\in E}\{h(x)+I(x)\}.$$
\end{definition}

Given a probabilty space $(\Omega,\mathcal{F},P)$, the random variables $\{Z_\varepsilon\}$ and $\{\overline{Z}_\varepsilon\}$ which take values in $(E,\rho)$ are called exponentially equivalent if for each $\delta>0$, 
$$\lim_{\varepsilon\rightarrow 0}\varepsilon\log P(\rho(Z_\varepsilon,\overline{Z}_\varepsilon)>\delta)=-\infty.$$
\vskip.10in
\begin{lemma}[{\cite[Theorem 4.2.13]{DZ98}}]\label{EXEQ}
If an LDP with a rate function $I(\cdot)$ holds for the random variables $\{Z_\varepsilon\}$, which are exponentially equivalent to $\{\overline{Z}_\varepsilon\}$, then the same LDP holds for $\{\overline{Z}_\varepsilon\}$.
\end{lemma}

 \vskip.10in
{\textbf{Existence and uniqueness of  solutions}}

We introduce the precise assumptions on the diffusion coefficient $\sigma$. Given a complete probability space $(\Omega, \mathcal{F}, P)$ with filtration $\{\mathcal{F}_t\}_{t\geqslant 0}$.  Let $L_2(l^2,U)$ denotes the Hilbert-Schmidt norms from $l^2$ to $U$ for a Hilbert space $U$.  We recall the following conditions for $\sigma$ from \cite{LZZ18}:

(i) \textbf{Growth condition}

There exists nonnegative constants $K'_i$, $K_i$, $\tilde{K}_i$ ($i=0,1,2$) such that for every $t\in[0,T]$:
 \vskip.10in
(A0) $\|\sigma(t,u)\|^2_{L_2(l^2, H^{-1})}\leqslant K_0'+K_1'\|u\|^2_{H}$;

(A1) $\|\sigma(t,u)\|^2_{L_2(l^2, H)}\leqslant K_0+K_1\|u\|^2_{H}+K_2\|\partial_1 u\|^2_H$;

(A2) $\|\sigma(t,u)\|^2_{L_2(l^2, H^{0,1})}\leqslant \tilde{K}_0+\tilde{K}_1\|u\|^2_{H^{0,1}}+\tilde{K}_2(\|\partial_1 u\|_H^2+\|\partial_1\partial_2u\|^2_H)$;

 \vskip.10in
(ii)\textbf{Lipschitz condition}

There exists nonnegative constants $L_1, L_2$ such that:
 \vskip.10in
(A3) $\|\sigma(t,u)-\sigma(t,v)\|^2_{L_2(l^2, H)}\leqslant L_1\|u-v\|^2_{H}+L_2\|\partial_1(u-v)\|^2_H$.

 \vskip.10in

The following theorem from \cite{LZZ18} shows the well-posedness of equation (\ref{equation after projection}):

\begin{Thm}[{\cite[Theorem 4.1, Theorem 4.2]{LZZ18}}]\label{ex. and uniq. of solution}
Under the assumptions (A0), (A1), (A2) and (A3) with $K_2<\frac{2}{21}, \tilde{K}_2<\frac{1}{5}, L_2<\frac{1}{5}$, equation (\ref{equation after projection}) has a unique probabilistically strong solution $u\in L^\infty([0,T], \tilde{H}^{0,1})\cap L^2([0, T], \tilde{H}^{1,1})\cap C([0,T], H^{-1})$ for $u_0\in\tilde{H}^{0,1}$. 
\end{Thm}

\section{Central limit theorem}

In this section, we will establish the central limit theorem. Let $u^\varepsilon$ be the solution to (\ref{equation after projection with varepsilon}) and $u^0$ the solution to (\ref{determinstic equation after projection}). Then we have the  following estimates from Lemma 3.5, Lemma 4.1, Lemma 4.2 and Lemma 4.4 in \cite{LZZ18}:

\begin{lemma}\label{estimates from LZZ18}
Assume (A0)-(A3) hold with $K_2<\frac{2}{21}, \tilde{K}_2<\frac{1}{5}, L_2<\frac{1}{5}$, there exists $\varepsilon_0>0$ such that 
$$\sup_{\varepsilon\in(0,\varepsilon_0)}E\left(\sup_{t\in[0,T]}\|u^\varepsilon(t)\|^2_H+\int^T_0\|u^\varepsilon(s)\|^2_{\tilde{H}^{1,0}}ds\right)\leqslant C.$$
Particularly,
$$\sup_{t\in[0,T]}\|u^0(t)\|^2_{\tilde{H}^{0,1}}+\int^T_0\|u^0(s)\|^2_{\tilde{H}^{1,1}}ds\leqslant C.$$
\end{lemma}

We have the following $\tilde{H}^{0,2}$ estimate for $u^0$:

\begin{lemma}\label{wellposedness result for u0 in H02}
Given  $u_0\in \tilde{H}^{0,2}$, the unique solution $u^0$ to  (\ref{determinstic equation after projection}) satisfies the following estimate:
 \begin{equation}\label{apriori estimate H02}
\sup_{t\in[0,T]}\|u^0(t)\|^2_{\tilde{H}^{0,2}}+\int^T_0\| u^0(t)\|^2_{\tilde{H}^{1,2}}dt\leqslant  C.
\end{equation}
\end{lemma}

\begin{proof}
Let's start by proving a priori estimates for $u^0$. Applying the operator  $\Delta^v_k$ and using an $L^2$ energy estimate, we have
\begin{align*}
\frac{1}{2}\frac{\text{d}}{\text{d}t}\|u^0_k(t)\|^2_H+\|\partial_1u^0_k(t)\|^2_H\leqslant \langle \Delta^v_k (u^0\cdot \nabla u^0), u^0_k\rangle,
\end{align*}
where we denote by $u^0_k$ the term $\Delta^v_ku^0$. By  Lemma \ref{commutator estimates} with $s=2, s_0=1$ and $u=v=u^0$, there exists $d_k\in l^1$ such that 
\begin{align*}
&\frac{1}{2}\frac{\text{d}}{\text{d}t}\|u^0_k(t)\|^2_H+\|\partial_1u^0_k(t)\|^2_H\\
\leqslant &Cd_k2^{-4k} \left(\|u^0\|_{\tilde{H}^{\frac{1}{4}, 2}}\|u^0\|_{\tilde{H}^{\frac{1}{4},1}}\|\partial_1u^0\|_{\tilde{H}^{0,2}}+\|u^0\|^2_{\tilde{H}^{\frac{1}{4},2}}\|\partial_1u^0\|_{\tilde{H}^{0,1}}\right).
\end{align*}

Now multiplying by $2^{4k}$ and taking sum over $k$ gives
\begin{align*}
\frac{1}{2}\frac{\text{d}}{\text{d}t}\|u^0(t)\|_{\tilde{H}^{0,2}}^2+\|\partial_1u^0(t)\|^2_{\tilde{H}^{0,2}}\leqslant C\left(\|u^0\|_{\tilde{H}^{\frac{1}{4}, 2}}\|u^0\|_{\tilde{H}^{\frac{1}{4},1}}\|\partial_1u^0\|_{\tilde{H}^{0,2}}+\|u^0\|^2_{\tilde{H}^{\frac{1}{4},2}}\|\partial_1u^0\|_{\tilde{H}^{0,1}}\right).
\end{align*}

By interpolation inequalities (see \cite[Theorem 2.80]{BCD11}) we have
\begin{align*}
\|u^0\|_{\tilde{H}^{\frac{1}{4}, s}}\leqslant &\|u^0\|^{\frac{3}{4}}_{\tilde{H}^{0, s}}\|u^0\|^{\frac{1}{4}}_{\tilde{H}^{1,s}},
\end{align*}
where $s=1,2$. Thus  we infer that
\begin{align*}
&\frac{1}{2}\frac{\text{d}}{\text{d}t}\|u^0(t)\|_{\tilde{H}^{0,2}}^2+\|\partial_1u^0(t)\|^2_{\tilde{H}^{0,2}}\\
\leqslant & C\Big(\|u^0\|^{\frac{3}{4}}_{\tilde{H}^{0, 2}}\|u^0\|_{\tilde{H}^{\frac{1}{4},1}}\|\partial_1u^0\|_{\tilde{H}^{0,2}}^{\frac{5}{4}}+\|u^0\|_{\tilde{H}^{0, 2}}\|u^0\|_{\tilde{H}^{\frac{1}{4},1}}\|\partial_1u^0\|_{\tilde{H}^{0,2}}\\
&+\|u^0\|^{\frac{3}{2}}_{\tilde{H}^{0,2}}\|\partial_1u^0\|^{\frac{1}{2}}_{\tilde{H}^{0,2}}\|\partial_1u^0\|_{\tilde{H}^{0,1}}+\|u^0\|^{2}_{\tilde{H}^{0,2}}\|\partial_1u^0\|_{\tilde{H}^{0,1}}\Big)\\
\leqslant &\alpha \|\partial_1 u^0\|^2_{\tilde{H}^{0,2}}+C\|u^0\|^{\frac{8}{3}}_{\tilde{H}^{\frac{1}{4},1}}\|u^0\|^2_{\tilde{H}^{0,2}}+C\|u^0\|^{2}_{\tilde{H}^{\frac{1}{4},1}}\|u^0\|^2_{\tilde{H}^{0,2}}\\
&+C\|\partial_1 u^0\|^{\frac{4}{3}}_{\tilde{H}^{0,1}}\|u^0\|^2_{\tilde{H}^{0,2}}+\|\partial_1 u^0\|_{\tilde{H}^{0,1}}\|u^0\|^2_{\tilde{H}^{0,2}}\\
\leqslant &\alpha \|\partial_1 u^0\|^2_{\tilde{H}^{0,2}}+C\|u^0\|^2_{\tilde{H}^{0,1}}\|u^0\|^{\frac{2}{3}}_{\tilde{H}^{1,1}}\|u^0\|^2_{\tilde{H}^{0,2}}+C\|u^0\|^{\frac{3}{2}}_{\tilde{H}^{0,1}}\|u^0\|^{\frac{1}{2}}_{\tilde{H}^{1,1}}\|u^0\|^2_{\tilde{H}^{0,2}}\\
&+C\|\partial_1 u^0\|^{\frac{4}{3}}_{\tilde{H}^{0,1}}\|u^0\|^2_{\tilde{H}^{0,2}}+\|\partial_1 u^0\|_{\tilde{H}^{0,1}}\|u^0\|^2_{\tilde{H}^{0,2}}\\
\leqslant &\alpha \|\partial_1 u^0\|^2_{\tilde{H}^{0,2}}+C(1+\|u^0\|^2_{\tilde{H}^{0,1}})(1+\|u^0\|^{2}_{\tilde{H}^{1,1}})\|u^0\|^2_{\tilde{H}^{0,2}},
\end{align*}
where we used Young's inequality in the third inequality and $\alpha<\frac{1}{2}$. Then Gronwall's inequality implies that
\begin{align*}
&\sup_{t\in[0,T]}\|u^0(t)\|^2_{\tilde{H}^{0,2}}+\int^T_0\|\partial_1 u^0(t)\|^2_{\tilde{H}^{0,2}}dt\\
\leqslant &\|u_0\|^2_{\tilde{H}^{0,2}}\exp\left(C\sup_{t\in[0,T]}(1+\|u^0(t)\|^2_{\tilde{H}^{0,1}})\int^T_0(1+\|u^0(t)\|^2_{\tilde{H}^{1,1}})dt\right).
\end{align*}

Then by Lemma \ref{estimates from LZZ18}, we get the result.

\end{proof}

The next proposition is about the convergence of $u^\varepsilon$.

\begin{prop}\label{estimate for ue-u0}
Assume (A0)-(A3) hold with $K_2<\frac{2}{21}, \tilde{K}_2<\frac{1}{5}, L_2<\frac{1}{5}$,  then there exists a constant $\varepsilon_0>0$ such that, for any $\varepsilon\in(0,\varepsilon_0)$, we have
\begin{equation}\aligned
E\left(\sup_{t\in[0,T]}\|u^\varepsilon(t)-u^0(t)\|^2_H+\int^T_0\|u^\varepsilon(s)-u^0(s)\|^2_{\tilde{H}^{1,0}}ds\right)\leqslant  C\varepsilon.
\endaligned
\end{equation}
\end{prop}

\begin{proof}
Applying It\^o's formula to $\|u^\varepsilon(t)-u^0(t)\|^2_H$, we have
\begin{align*}
&\|u^\varepsilon(t)-u^0(t)\|^2_H\\
=&-2\int^t_0\|\partial_1(u^\varepsilon-u^0)(s)\|^2_Hds-2\int^t_0\langle u^\varepsilon(s)-u^0(s), B(u^\varepsilon(s))-B(u^0(s))\rangle ds\\
&+2\sqrt{\varepsilon}\int^t_0\langle u^\varepsilon(s)-u^0(s), \sigma(s,u^\varepsilon(s))dW(s)\rangle +\varepsilon\int^t_0 \|\sigma(s,u^\varepsilon(s))\|^2_{L_2(l^2,H)}ds.
\end{align*}

By Lemma \ref{b(u,v,w)} we have
\begin{align*}
&|\langle u^\varepsilon(s)-u^0(s), B(u^\varepsilon(s))-B(u^0(s))\rangle|\\
=&|b(u^\varepsilon,u^\varepsilon,u^\varepsilon-u^0)-b(u^0, u^0,u^\varepsilon-u^0)|\\
=& |b(u^\varepsilon-u^0,u^0,u^\varepsilon-u^0)|\\
\leqslant &\frac{1}{4} \|\partial_1(u^\varepsilon-u^0)\|^2_H+C(1+\|u^0\|^2_{\tilde{H}^{1,1}})\|u^\varepsilon-u^0\|^2_H.
\end{align*}
By the Burkh\"older-Davis-Gundy's inequality (see \cite[Appendix D]{LR15}), we have
\begin{align*}
&2\sqrt{\varepsilon} E\left(\sup_{s\in[0,t]}\left|\int^t_0\langle u^\varepsilon(s)-u^0(s), \sigma(s,u^\varepsilon(s))dW(s)\rangle\right|\right)\\
\leqslant& 6\sqrt{\varepsilon} E\left(\int^t_0\|u^\varepsilon(s)-u^0(s)\|^2_H\|\sigma(s,u^\varepsilon(s)\|^2_{L_2(l^2,H)}ds\right)^\frac{1}{2}\\
\leqslant& 6\sqrt{\varepsilon} E\left(\sup_{s\in[0,t]}\|u^\varepsilon(s)-u^0(s)\|^2_H\int^t_0(K_0+K_1\|u^\varepsilon(s)\|^2_H+K_2\|\partial_1u^\varepsilon(s)\|^2_H)ds\right)^\frac{1}{2}\\
\leqslant & \frac{1}{2} E\left(\sup_{s\in[0,t]}\|u^\varepsilon(s)-u^0(s)\|^2_H\right)+C\varepsilon E\left(\int^t_0(1+\|u^\varepsilon(s)\|^2_H+\|\partial_1u^\varepsilon(s)\|^2_H)ds\right),
\end{align*}
where we used (A1) in the last second line. Thus by above estimates and (A1) we deduce that
\begin{align*}
&E\left(\sup_{s\in[0,t]}\|u^\varepsilon(s)-u^0(s)\|^2_H+\int^t_0 \|u^\varepsilon(s)-u^0(s)\|^2_{\tilde{H}^{1,0}}ds\right)\\
\leqslant & C\int^t_0(1+\|u^0(s)\|^2_{\tilde{H}^{1,1}})E(\sup_{l\in[0,s]}\|u^\varepsilon(l)-u^0(l)\|^2_H)ds\\
&+C\varepsilon E\left(\int^t_0(1+\|u^\varepsilon(s)\|^2_H+\|\partial_1u^\varepsilon(s)\|^2_H)ds\right).
\end{align*}
Then Gronwall's inequality and Lemma \ref{estimates from LZZ18} imply that
\begin{align*}
&E\left(\sup_{s\in[0,T]}\|u^\varepsilon(s)-u^0(s)\|^2_H+\int^T_0 \|u^\varepsilon(s)-u^0(s)\|^2_{\tilde{H}^{1,0}}ds\right)\\
\leqslant & C\varepsilon E\left(\int^T_0(1+\|u^\varepsilon(s)\|^2_H+\|\partial_1u^\varepsilon(s)\|^2_H)ds\right) e^{C\int^T_0(1+\|u^0(s)\|^2_{\tilde{H}^{1,1}})ds}\\
\leqslant & C\varepsilon.
\end{align*}

\end{proof}

Let $V^0$ be the solution to the following SPDE:
\begin{equation}\label{limit eq. of CLT}\aligned
dV^0(t)&=\partial_1^2V^0(t)dt-B(V^0(t),u^0(t))dt-B(u^0(t),V^0(t))dt+\sigma(t,u^0(t))dW(t),\\
V^0(0)&=0.
\endaligned
\end{equation}

\begin{lemma}\label{wellposedness limit eq. of CLT}
Assume that $u^0$ satisfies \eqref{apriori estimate H02}. Then under the assumptions (A0), (A1), (A2), equation  \eqref{limit eq. of CLT} has a unique probabilistically strong solution $$V^0\in L^\infty([0,T], \tilde{H}^{0,1})\cap L^2([0, T], \tilde{H}^{1,1})\cap C([0,T], H^{-1}).$$
\end{lemma}

\begin{proof}
The proof follows a very similar Galerkin approximation argument as in \cite[Section 4]{LZZ18}, we show some key steps here.

Let $\{e_k, k\geqslant 1\}$ be an orthonormal basis of $H$ whose elements belong to $H^2$ and orthogonal in $\tilde{H}^{0,1}$ and $\tilde{H}^{1,0}$. Let $\mathcal{H}_n=\text{span} \{e_1, \dots, e_n\}$ and let $P_n$ denote the orthogonal projection from $H$ to $\mathcal{H}_n$. For $l^2-$cylindrical Wiener process $W(t)$, let $W_n(t)=\Pi_nW(t):=\sum_{j=1}^n\psi_j\beta_j(t)$, where $\beta_j$ is a sequence of independent Brownian motions and $\psi_j$ is an orthonormal basis of  $l^2$. Set $F: H^1\rightarrow H^{-1}$ with $F(u)=-B(u,u^0)-B(u^0, u)+\partial_1^2u$.

Fix $n\geqslant 1$ and for $v\in \mathcal{H}_n$ consider the following equation on $\mathcal{H}_n$:
\begin{equation}\label{Galerkin eq.}\aligned
d\langle V_n(t), v\rangle =&\langle P_nF(V_n), v\rangle dt+\langle P_n\sigma(t,u^0(t))dW_n(t), v\rangle\\
V_n(0)=& P_nu_0.
\endaligned
\end{equation}

Then by \cite[Theorem 3.1.1]{LR15} there exists unique global strong solution $V_n$ to \eqref{Galerkin eq.}. Moreover, $V_n\in C([0,T], \mathcal{H}_n)$.

We first prove  a priori estimates. 
Applying It\^o's formula to $\|V_n\|^2_{\tilde{H}^{0,1}}$, we have
\begin{align*}
\|V_n(t)\|^2_{\tilde{H}^{0,1}}+2\int^t_0\|\partial_1V_n(s)\|^2_{\tilde{H}^{0,1}}ds=&\|P_nu_0\|^2_{\tilde{H}^{0,1}}-2\int^t_0 \langle B(V_n, u^0)+ B(u^0,V_n), V_n \rangle_{\tilde{H}^{0,1}} ds\\
&+2\int^t_0\langle \sigma(s,u^0(s))dW_n(s), V_n(s) \rangle_{\tilde{H}^{0,1}} \\
&+\int^t_0 \|P_n\sigma(s, u^0(s))\Pi_n\|^2_{L_2(l^2, \tilde{H}^{0,1})}ds.
\end{align*}

By Lemma \ref{b(u,v,w)} and Young's inequality, we have
\begin{align*}
&|\langle B(V_n, u^0)+ B(u^0,V_n), V_n \rangle_{\tilde{H}^{0,1}}|\\
\leqslant & |b(V_n,u^0, V_n)|+|b(\partial_2V_n,u^0,\partial_2V_n)|+|b(V_n, \partial_2u^0, \partial_2V_n)|+|b(\partial_2u^0,V_n,\partial_2V_n)|\\
\leqslant &C\Big(\|V_n\|_{\tilde{H}^{1,0}}\|u^0\|_{\tilde{H}^{1,1}}\|V_n\|_{H}+\|\partial_2V_n\|_{\tilde{H}^{1,0}}\|u^0\|_{\tilde{H}^{1,1}}\|\partial_2V_n\|_{H}\\
&+\|V_n\|_{\tilde{H}^{1,0}}\| \partial_2u^0\|_{\tilde{H}^{1,1}}\|\partial_2V_n\|_{H}+\|\partial_2u^0\|_{\tilde{H}^{1,0}}\|V_n\|_{\tilde{H}^{1,1}}\|\partial_2V_n\|_H  \Big)\\
\leqslant & \alpha \|V_n\|^2_{\tilde{H}^{1,1}}+C\|u^0\|^2_{\tilde{H}^{1,2}}\|V_n\|^2_{\tilde{H}^{0,1}},
\end{align*}
where $\alpha<\frac{1}{2}$.

The growth condition and Lemma \ref{estimates from LZZ18} imply that
\begin{align*}
\int^t_0 \|P_n\sigma(s, u^0(s))\Pi_n\|^2_{L_2(l^2, \tilde{H}^{0,1})}ds\leqslant C\int^t_0(1+\|u^0\|^2_{\tilde{H}^{1,1}})ds\leqslant C.
\end{align*}

Similarly, by the Burkh\"older-Davis-Gundy's inequality, we have
\begin{align*}
&2 E\left(\sup_{s\in[0,t]}\left|\int^t_0\langle \sigma(s,u^0(s))dW_n(s), V_n(s) \rangle_{\tilde{H}^{0,1}} \right|\right)\\
\leqslant& 6 E\left(\int^t_0 \|P_n\sigma(s, u^0(s))\Pi_n\|^2_{L_2(l^2, \tilde{H}^{0,1})}\|V_n(s)\|^2_{\tilde{H}^{0,1}}ds\right)^\frac{1}{2}\\
\leqslant &\beta  E\left(\sup_{s\in[0,t]}\|V^0(s)\|^2_{\tilde{H}^{0,1}}\right)+C\int^t_0(1+\|u^0\|^2_{\tilde{H}^{1,1}})ds\\
\leqslant &\beta  E\left(\sup_{s\in[0,t]}\|V^0(s)\|^2_{\tilde{H}^{0,1}}\right)+C,
\end{align*}
where $\beta<\frac{1}{2}$.

Then we get
\begin{align*}
&E\left(\sup_{s\in[0,t]}\|V_n(s)\|^2_{\tilde{H}^{0,1}}\right)+E\int^t_0\|V_n(s)\|^2_{\tilde{H}^{1,1}}ds\\
\leqslant& C+C\int^t_0\left(\|u^0\|^2_{\tilde{H}^{1,2}}+1\right)E\left(\sup_{r\in[0,s]}\|V_n(r)\|^2_{\tilde{H}^{0,1}}\right)ds.
\end{align*}

Then by Gronwall's inequality and \eqref{apriori estimate H02}, we have
\begin{equation}\label{H01 estimate for V0}\aligned
E\left(\sup_{s\in[0,t]}\|V_n(s)\|^2_{\tilde{H}^{0,1}}\right)+E\int^t_0\|V_n(s)\|^2_{\tilde{H}^{1,1}}ds
\leqslant & C\exp\left(C\int^t_0\left(\|u^0\|^2_{\tilde{H}^{0,2}}+1\right)ds\right)\leqslant C.
\endaligned
\end{equation}

The rest part  of the existence proof is very similar as in the proof of \cite[Theorem 4.1]{LZZ18}, we only need to point out that the convergence of $F(V_n)$ holds as $n\rightarrow\infty$: 
From the proof we could obtain that there exists another stochastic basis $(\tilde{\Omega}, \tilde{\mathcal{F}},\tilde{P})$ and random variables $\tilde{V}_n$ with same law of $V_n$ such that $\tilde{V}_n\rightarrow \tilde{V}$ in $C([0,T],H^{-1})\cap L^2([0,T],H)$, $\tilde{P}$-a.s. (in the sense of subsequence).  Fix $l\in C^\infty(\mathbb{T}^2)$ with $\text{div}l=0$. Since $F(V_n)$ is actually linear term, the convergence of $\tilde{V}_n$ in $L^2([0,T],H)$ implies that
$$\int^t_0\langle F(\tilde{V}_n), P_n l \rangle ds\rightarrow \int^t_0\langle F(\tilde{V}), l \rangle  ds, \tilde{P}\text{-a.s.}$$

For uniqueness, assume $V^0_1, V^0_2$ are two solutions  in $L^\infty([0,T], \tilde{H}^{0,1})\cap L^2([0, T], \tilde{H}^{1,1})\cap C([0,T], H^{-1})$ with the same initial condition, let $w=V_1-V_2$, then $w(0)=0$ and $w$ satisfies
\begin{align*}
dw(t)=\partial_1^2 w(t)dt-B(w(t),u^0(t))dt-B(u^0(t), w(t))dt.
\end{align*}

Then similarly as the proof of the uniqueness for the deterministic Navier-Stokes equation with anisotropic viscosity, we know that $w=0$.

\end{proof}

\begin{remark}
Note here we do not need assumption (A3) and $L^4(\Omega)$ estimate of $V_n$ since the drift term $\sigma(t,u^0)$ does not depend on $V_n$.
\end{remark}

The main theorem of this section is the following central limit theorem.

\begin{Thm}\label{CLT}
Assume (A0)-(A3) hold with $K_2<\frac{2}{21}, \tilde{K}_2<\frac{1}{5}, L_2<\frac{1}{5}$, then for $u_0\in\tilde{H}^{0,2}$ we have
\begin{align*}
\lim_{\varepsilon\rightarrow 0}E\left(\sup_{t\in[0,T]} \|\frac{u^\varepsilon(t)-u^0(t)}{\sqrt{\varepsilon}}-V^0(t)\|_H^2+\int^T_0\|\frac{u^\varepsilon(t)-u^0(t)}{\sqrt{\varepsilon}}-V^0(t)\|^2_{\tilde{H}^{1,0}}dt \right)=0
\end{align*}
\end{Thm}

\begin{proof}
Let $V^\varepsilon=\frac{u^\varepsilon(t)-u^0(t)}{\sqrt{\varepsilon}}$. Then we have
\begin{equation}\label{eq. for Vvarepsilon}\aligned
dV^\varepsilon(t)&=\partial_1^2V^\varepsilon(t)dt-B(V^\varepsilon(t),u^\varepsilon(t))dt-B(u^0(t),V^\varepsilon(t))dt+\sigma(t,u^\varepsilon(t))dW(t),\\
V^\varepsilon(0)&=0,
\endaligned
\end{equation}
and
\begin{align*}
d(V^\varepsilon-V^0)=&\partial_1^2(V^\varepsilon- V^0)dt-(B(V^\varepsilon,u^\varepsilon)-B(V^0,u^0))dt\\
&-B(u^0,V^\varepsilon- V^0)dt+(\sigma(t,u^\varepsilon)-\sigma(t,u^0))dW(t).
\end{align*}

By It\^o's formula, we have
\begin{align*}
&\|V^\varepsilon(t)-V^0(t)\|^2_H+2\int^t_0\|\partial_1(V^\varepsilon(s)-V^0(s))\|^2_Hds\\
=&-2\int^t_0\langle B(V^\varepsilon,u^\varepsilon)-B(V^0,u^0), V^\varepsilon-V^0 \rangle ds\\
&+2\int^t_0\langle (\sigma(s,u^\varepsilon)-\sigma(s,u^0))dW(s), V^\varepsilon(s)-V^0(s)\rangle\\
&+\int^t_0\|\sigma(s,u^\varepsilon)-\sigma(s,u^0)\|^2_{L_2(l^2,H)}ds\\
\leqslant &2\int^t_0|b(V^\varepsilon-V^0, u^0, V^\varepsilon-V^0)|ds\\
&+2\int^t_0|b(V^\varepsilon, u^\varepsilon-u^0, V^\varepsilon-V^0)|ds\\
&+2|\int^t_0\langle (\sigma(s,u^\varepsilon)-\sigma(s,u^0))dW(s), V^\varepsilon(s)-V^0(s)\rangle|\\
&+\int^t_0\|\sigma(s,u^\varepsilon)-\sigma(s,u^0)\|^2_{L_2(l^2,H)}ds\\
=:&I_1+I_2+I_3+I_4.
\end{align*}

Taking the supremum and the expectation, we obtain that
\begin{align*}
&E\left(\sup_{s\in[0,t]}\|V^\varepsilon(s)-V^0(s)\|^2_H+2\int^t_0\|\partial_1(V^\varepsilon(s)-V^0(s))\|^2_{H}ds\right)\\
\leqslant & E(I_1(t)+I_2(t)+\sup_{s\in[0,t]}I_3(s)+I_4(t)).
\end{align*}

By Lemma \ref{b(u,v,w)}, we have
\begin{align*}
EI_1(t)\leqslant & 2E\int^t_0\left(\frac{1}{4}\|V^\varepsilon-V^0\|^2_{\tilde{H}^{1,0}}+C\|u^0\|^2_{\tilde{H}^{1,1}}\|V^\varepsilon-V^0\|^2_H\right)ds.
\end{align*}

By Lemma \ref{b(u,v,w)}, we have
\begin{align*}
EI_2(t)=&2\sqrt{\varepsilon}E\int^t_0|b(V^\varepsilon, V^\varepsilon, V^\varepsilon-V^0)|ds\\
=&2\sqrt{\varepsilon}E\int^t_0|b(V^\varepsilon, V^\varepsilon, V^0)|ds=2\sqrt{\varepsilon}E\int^t_0|b(V^\varepsilon, V^0, V^\varepsilon)|ds\\
\leqslant &\sqrt{\varepsilon}CE\int^t_0(\|V^\varepsilon\|^2_{\tilde{H}^{1,0}}\|V^\varepsilon\|^2_H+\|V^0\|^2_{\tilde{H}^{1,1}})ds .
\end{align*}

By the Burkh\"older-Davis-Gundy inequality and (A3), we have
\begin{align*}
E\left(\sup_{s\in[0,t]}I_3(s)\right)\leqslant & 6E\left(\int^t_0\|\sigma(s,u^\varepsilon)-\sigma(s,u^0)\|^2_{L_2(l^2,H)}\|V^\varepsilon-V^0\|^2_Hds\right)^\frac{1}{2}\\
\leqslant & 6E\left(\sup_{s\in[0,t]}\|V^\varepsilon-V^0\|^2_H\int^t_0\|\sigma(s,u^\varepsilon)-\sigma(s,u^0)\|^2_{L_2(l^2,H)}ds\right)^\frac{1}{2}\\
\leqslant &\frac{1}{2}E\left(\sup_{s\in[0,t]}\|V^\varepsilon-V^0\|^2_H\right)+CE\left(\int^t_0\|u^\varepsilon-u^0)\|^2_H+\|\partial_1(u^\varepsilon-u^0)\|^2_Hds\right).
\end{align*}

By (A1), we have
\begin{align*}
EI_4(t)\leqslant CE\left(\int^t_0\|u^\varepsilon-u^0\|^2_H+\|\partial_1(u^\varepsilon-u^0)\|^2_Hds\right).
\end{align*}

The above estimates together with Lemma \ref{estimate for ue-u0} and Lemma \ref{estimate for Ve} below induce that
\begin{align*}
&E\left(\sup_{s\in[0,t]}\|V^\varepsilon(s)-V^0(s)\|^2_H+\int^t_0\|V^\varepsilon(s)-V^0(s)\|^2_{\tilde{H}^{1,0}}ds\right)\\
\leqslant& C E\int^t_0\left(\|u^0(s)\|^2_{\tilde{H}^{1,1}}\sup_{l\in[0,s]}\|V^\varepsilon(l)-V^0(l)\|^2_H\right)ds\\
&+\sqrt{\varepsilon}CE\int^t_0(\|V^\varepsilon\|^2_{\tilde{H}^{1,0}}\|V^\varepsilon\|^2_H+\|V^0\|^2_{\tilde{H}^{1,1}})ds\\
&+CE\left(\int^t_0\|u^\varepsilon-u^0\|^2_H+\|\partial_1(u^\varepsilon-u^0)\|^2_Hds\right)\\
\leqslant& C E\int^t_0\left((1+\|u^0(s)\|^2_{\tilde{H}^{1,1}})\sup_{l\in[0,s]}\|V^\varepsilon(l)-V^0(l)\|^2_H\right)ds+C(\sqrt{\varepsilon}+\varepsilon).
\end{align*}

Then by Gronwall's inequality and Lemma \ref{estimates from LZZ18} we have

\begin{align*}
&E\left(\sup_{s\in[0,t]}\|V^\varepsilon(s)-V^0(s)\|^2_H+\int^t_0\|V^\varepsilon(s)-V^0(s)\|^2_{\tilde{H}^{1,0}}ds\right)\\
\leqslant &C(\sqrt{\varepsilon}+\varepsilon)\exp\left(C\int^t_0(1+\|u^0(s)\|^2_{\tilde{H}^{1,1}})ds\right)\leqslant C(\sqrt{\varepsilon}+\varepsilon).
\end{align*}

Let $\varepsilon\rightarrow 0$, we complete the proof.

\end{proof}

It remains to establish the following lemma.

\begin{lemma}\label{estimate for Ve}
Assume (A0)-(A3) hold with $K_2<\frac{2}{21}, \tilde{K}_2<\frac{1}{5}, L_2<\frac{1}{5}$. Let $V^\varepsilon$ be the solution to \eqref{eq. for Vvarepsilon}, then  there exists a constant $\varepsilon_0>0$ such that
$$\sup_{\varepsilon\in(0,\varepsilon_0)}E\int^T_0\|V^\varepsilon(s)\|^2_H\|V^\varepsilon(s)\|^2_{\tilde{H}^{1,0}}ds<\infty.$$
\end{lemma}

\begin{proof}
Applying It\^o's formula to $\|V^\varepsilon\|^4_H$, we have
\begin{align*}
d\|V^\varepsilon\|^4_H\leqslant &2\|V^\varepsilon\|^2_H\Big(-2\|\partial_1V^\varepsilon\|^2_Hdt-2b(V^\varepsilon, u^\varepsilon, V^\varepsilon)dt\\
&+2\langle\sigma(t,u^\varepsilon)dW(t), V^\varepsilon\rangle+\|\sigma(t,u^\varepsilon)\|^2_{L_2(l^2,H)}dt\Big)+4\|\left(\sigma(t,u^\varepsilon(t))\right)^* V^\varepsilon\|^2_{l^2}dt.
\end{align*}

Taking the supremum and the expectation, we have
\begin{align*}
&E\left(\sup_{s\in[0,t]}\|V^\varepsilon(s)\|^4_H+4\int^t_0\|V^\varepsilon(s)\|^2_H\|\partial_1V^\varepsilon(s)\|^2_Hds\right)\\
\leqslant& 4E\left(\int^t_0\|V^\varepsilon(s)\|^2_H|b(V^\varepsilon(s),u^\varepsilon(s),V^\varepsilon(s))|ds\right)\\
&+6E\left(\int^t_0\|V^\varepsilon(s)\|^2_H\|\sigma(s,u^\varepsilon(s))\|^2_{L_2(l^2,H)}ds\right)\\
&+4E\left(\sup_{s\in[0,t]}\left|\int^t_0\|V^\varepsilon(s)\|^2_H\langle\sigma(s,u^\varepsilon(s))dW(s), V^\varepsilon(s)\rangle\right| \right)\\
=:&I_1+I_2+I_3.
\end{align*}

Recall that $V^\varepsilon=\frac{u^\varepsilon-u^0}{\sqrt{\varepsilon}}$.  By Lemma \ref{b(u,v,w)}, we have
\begin{align*}
I_1(t) =&  4E\left(\int^t_0\|V^\varepsilon(s)\|^2_H|b(V^\varepsilon(s), u^0(s)+\sqrt{\varepsilon}V^\varepsilon(s),V^\varepsilon(s))|ds\right)\\
=& 4E\left(\int^t_0\|V^\varepsilon(s)\|^2_H|b(V^\varepsilon(s), u^0(s),V^\varepsilon(s))|ds\right)\\
\leqslant & E\left(\int^t_0\|V^\varepsilon(s)\|^2_H(\|\partial_1V^\varepsilon(s)\|^2_H+C(1+\|u^0(s)\|^2_{\tilde{H}^{1,1}})\|V^\varepsilon(s)\|^2_H)ds\right)\\
\leqslant &E\int^t_0\|V^\varepsilon(s)\|^2_H\|\partial_1V^\varepsilon(s)\|^2_Hds+CE\left(\int^t_0(1+\|u^0(s)\|^2_{\tilde{H}^{1,1}})\sup_{l\in[0,s]}\|V^\varepsilon(l)\|^4_Hds\right).
\end{align*}

Note that Proposition \ref{estimate for ue-u0} implies the boundedness of $u^0$ in $L^2([0,T],\tilde{H}^{1,1})$. By (A1) we have
\begin{align*}
I_2(t)\leqslant & CE\left(\int^t_0\|V^\varepsilon(s)\|^2_H(1+\|u^\varepsilon(s)\|^2_H+\|\partial_1u^\varepsilon(s)\|^2_H)ds\right)\\
\leqslant &CE\left(\int^t_0\|V^\varepsilon(s)\|^2_H(1+\|u^0(s)\|^2_H +\varepsilon\|V^\varepsilon(s)\|^2_H+\|\partial_1u^0(s)\|^2_H+\varepsilon\|\partial_1V^\varepsilon(s)\|^2_H)ds\right)\\
\leqslant&C+\varepsilon CE\left(\sup_{s\in[0,t]}\|V^\varepsilon(s)\|^4_H\right)+\varepsilon C E\left(\int^t_0\|V^\varepsilon(s)\|^2_H\|\partial_1V^\varepsilon(s)\|^2_Hds\right).
\end{align*} 

By the Burkholder-Davis-Gundy inequality, (A1) and Proposition \ref{estimate for ue-u0}, we have
\begin{align*}
&I_3(t)\\
\leqslant &C E\left(\int^t_0\|V^\varepsilon(s)\|^6_H\|\sigma(s,u^\varepsilon(s))\|^2_{L_2(l^2,H)}ds\right)^\frac{1}{2}\\
\leqslant & C E\left(\sup_{s\in[0,t]}\|V^\varepsilon(s)\|^2_H\left(\int^t_0\|V^\varepsilon(s)\|^2_H(1+\|u^\varepsilon(s)\|^2_H+\|\partial_1u^\varepsilon(s)\|^2_H)ds\right)^\frac{1}{2}\right)\\
\leqslant& \frac{1}{2}E\left(\sup_{s\in[0,t]}\|V^\varepsilon(s)\|^4_H\right)\\
&+CE\left(\int^t_0\|V^\varepsilon(s)\|^2_H(1+\|u^0(s)\|^2_H +\varepsilon\|V^\varepsilon(s)\|^2_H+\|\partial_1u^0(s)\|^2_H+\varepsilon\|\partial_1V^\varepsilon(s)\|^2_H)ds\right)\\
\leqslant &(\frac{1}{2}+\varepsilon C)E\left(\sup_{s\in[0,t]}\|V^\varepsilon(s)\|^4_H\right)+C+\varepsilon CE\left(\int^t_0\|V^\varepsilon(s)\|^2_H\|\partial_1V^\varepsilon(s)\|^2_Hds\right).
\end{align*}

Combining the above estimates, there exists constants $C_0$ and $C_1$,
\begin{align*}
&E\left((\frac{1}{2}-C_0\varepsilon)\sup_{s\in[0,t]}\|V^\varepsilon(s)\|^2_H+(3-C_1\varepsilon)\int^t_0\|V^\varepsilon(s)\|^2_H\|\partial_1V^\varepsilon(s)\|^2_Hds\right)\\
\leqslant &C+CE\left(\int^t_0(1+\|u^0(s)\|^2_{\tilde{H}^{1,1}})\sup_{l\in[0,s]}\|V^\varepsilon(l)\|^4_Hds\right).
\end{align*}

When $\varepsilon<\varepsilon_0:=\min\{\frac{1}{4C_0}, \frac{3}{2C_1}\}$, by Gronwall's inequality, we have
\begin{align*}
E\left(\sup_{s\in[0,t]}\|V^\varepsilon(s)\|^4_H+\int^t_0\|V^\varepsilon(s)\|^2_H\|\partial_1V^\varepsilon(s)\|^2_Hds\right)\leqslant C\exp\left(\int^t_0(1+\|u^0(s)\|^2_{\tilde{H}^{1,1}})ds\right).
\end{align*}

Again by Lemma \ref{estimates from LZZ18} we complete the proof.

\end{proof}

\section{Moderate deviations}

 In this section, we will prove that $Z^\varepsilon:=\frac{1}{\sqrt{\varepsilon}\lambda(\varepsilon)}(u^\varepsilon-u^0)$ satisfies LDP on $$L^\infty([0,T], H)\cap L^2([0,T], \tilde{H}^{1,0})\cap C([0,T],H^{-1})$$ if $\lambda(\varepsilon)$ satisfies:
$$\lambda(\varepsilon)\rightarrow\infty, \text{ }\text{ }\text{ }\sqrt{\varepsilon}\lambda(\varepsilon)\rightarrow 0\text{ as }\varepsilon\rightarrow0.$$

  We will use the weak convergence approach introduced by Budhiraja and Dupuis in \cite{BD00}.  The starting point is the equivalence between the large deviation principle and the Laplace principle. This result was first formulated in \cite{P93} and it is essentially a consequence of Varadhan's lemma \cite{V66} and Bryc's converse theorem \cite{B90}.

\begin{remark}
By \cite{DZ98} we have the the equivalence between the large deviation principle and the Laplace principle in completely regular topological spaces. In \cite{BD00} the authors give the weak convergence approach on a Polish space.  Since the  proof  does not depend on the separability and the completeness, the result also holds in metric spaces.
\end{remark}

Let $\{W(t)\}_{t\geqslant 0}$ be a cylindrical Wiener process on $l^2$ w.r.t. a complete filtered probability space $(\Omega, \mathcal{F}, \mathcal{F}_t, P)$ (i.e. the path of $W$ take values in $C([0,T]; U)$, where $U$ is another Hilbert space such that the embedding $l^2\subset U$ is Hilbert-Schmidt). For $\varepsilon>0$, suppose $g^\varepsilon$: $C([0,T], U)\rightarrow E$ is a measurable map. 
Let

$$\mathcal{A}:=\left\{v: v \text{ is }l^2\text{-valued }\mathcal{F}_t\text{-predictable process and }\int^T_0\|v(s)(\omega)\|^2_{l^2}ds<\infty\text{ a.s.}\right\},$$
$$S_N:=\left\{\phi\in L^2([0,T],l^2):\text{ }\int^T_0\|\phi(s)\|^2_{l^2}ds\leqslant N\right\},$$
$$\mathcal{A}_N:=\left\{v\in\mathcal{A}:\text{ } v(\omega)\in S_N\text{ P-a.s.}\right\}.$$
Here we will always refer to the weak topology on $S_N$ in the following if we do not state it explicitly.

Now we formulate the following sufficient conditions for the Laplace principle of $g^\varepsilon(W(\cdot))$ as $\varepsilon\rightarrow 0$.

\begin{hyp}\label{Hyp}
There exists a measurable map $g^0: C([0,T], U)\rightarrow E$ such that the following two conditions hold:\\
1. Let $\{v^\varepsilon:\varepsilon>0\}\subset  \mathcal{A}_N$ for some $N<\infty$. If $v^\varepsilon$ converge to $v$ in distribution as $S_N$-valued random elements, then 
$$g^\varepsilon\left(W(\cdot)+\frac{1}{\sqrt{\varepsilon}}\int^\cdot_0v^\varepsilon(s)ds\right)\rightarrow g^0\left(\int^\cdot_0v(s)ds\right)$$
in distribution as $\varepsilon\rightarrow 0$.\\
2. For each $N<\infty$, the set
$$K_N=\left\{g^0\left(\int^\cdot_0\phi(s)ds\right): \phi\in S_N\right\}$$
is a compact subset of $E$. 
\end{hyp}

\begin{lemma}[{\cite[Theorem 4.4]{BD00}}]\label{weak convergence method}
If $u^\varepsilon=g^\varepsilon(W)$ satisfies the Hypothesis \ref{Hyp}, then the family $\{u^\varepsilon\}$ satisfies the Laplace principle (hence large deviation principle) on $E$ with the good rate function $I$ given by
\begin{equation}
I(f)=\inf_{\{\phi\in L^2([0,T], l^2):f=g^0(\int^\cdot_0\phi(s)ds)\}}\left\{\frac{1}{2}\int^T_0\|\phi(s)\|^2_{l^2}ds\right\}.
\end{equation}
\end{lemma}

Let us introduce the following skeleton equation associated to $Z^\varepsilon=\frac{1}{\sqrt{\varepsilon}\lambda(\varepsilon)}(u^\varepsilon-u^0)$, for $\phi\in L^2([0,T], l^2)$:
\begin{equation}\label{skeleton eq. MDP}\aligned
&dX^\phi(t)=\partial^2_1 X^\phi(t)dt-B(X^\phi(t), u^0(t))dt-B(u^0(t),X^\phi(t))dt+\sigma(t, u^0(t))\phi(t)dt,\\
&X^\phi(0)=0.
\endaligned
\end{equation}

Define $g^0:C([0,T],U)\rightarrow L^\infty([0,T], H)\bigcap L^2([0,T], \tilde{H}^{1,0})\bigcap C([0,T],H^{-1})$ by
\begin{align*}
g^0(h):=
\left\{
             \begin{array}{ll}
             X^\phi, &\text{ if }h=\int^\cdot_0\phi(s)ds \text{ for some }\phi\in L^2([0,T],l^2);  \\
             0, &\text{  otherwise.} 
             \end{array}
\right.
\end{align*}
Then the rate function can be written as
\begin{equation}\label{rate function LDP}
I(g)=\inf\left\{\frac{1}{2}\int^T_0\|\phi(s)\|^2_{l^2}ds: \text{ }g=X^\phi,\text{ }\phi\in L^2([0,T],l^2)\right\},
\end{equation}
where $g\in L^\infty([0,T], H)\bigcap L^2([0,T], \tilde{H}^{1,0})\bigcap C([0,T],H^{-1})$.

The main result of this section is the following one:
\begin{Thm}\label{main result MDP}
Assume (A0)-(A3) hold with $K_2<\frac{2}{21}, \tilde{K}_2<\frac{1}{5}, L_2<\frac{1}{5}$ and  $u_0\in \tilde{H}^{0,2}$, then $Z^\varepsilon$ satisfies a large deviation principle on $L^\infty([0,T], H)\bigcap L^2([0,T], \tilde{H}^{1,0})\bigcap C([0,T],H^{-1})$ with speed $\lambda^2(\varepsilon)$ and with the good rate function $I$ given by (\ref{rate function LDP}), more precisely, it holds that\\
(U) for all closed sets $F\subset L^\infty([0,T], H)\bigcap L^2([0,T], \tilde{H}^{1,0})\bigcap C([0,T],H^{-1})$ we have 
$$\limsup_{\varepsilon\rightarrow 0}\frac{1}{\lambda^2(\varepsilon)}\log P\left(\frac{u^\varepsilon-u^0}{\sqrt{\varepsilon}\lambda(\varepsilon)}\in F\right)\leqslant-\inf_{g\in F}I(g),$$
(L) for all open sets $G\subset  L^\infty([0,T], H)\bigcap L^2([0,T], \tilde{H}^{1,0})\bigcap C([0,T],H^{-1})$ we have 
$$\limsup_{\varepsilon\rightarrow 0}\frac{1}{\lambda^2(\varepsilon)}\log P\left(\frac{u^\varepsilon-u^0}{\sqrt{\varepsilon}\lambda(\varepsilon)}\in G\right)\geqslant-\inf_{g\in G}I(g).$$
\end{Thm}

By Lemma \ref{weak convergence method}, we should check that Hypothesis \ref{Hyp} holds with $\varepsilon$ replaced by $\lambda^{-2}$.  The proof is divided into the following lemmas. The first lemma is about the solution to \eqref{skeleton eq. MDP}.

\begin{prop}\label{wellposedness for skeleton eq. MDP}
Assume (A0)-(A2) hold. For all $u_0\in\tilde{H}^{0,2}$ and $\phi\in L^2([0,T], l^2)$ there exists a unique solution $$X^\phi\in L^\infty([0,T], \tilde{H}^{0,1})\bigcap L^2([0,T], \tilde{H}^{1,1})\bigcap C([0,T],H^{-1})$$ to (\ref{skeleton eq. MDP}).
\end{prop}

\begin{proof}
We start by giving a priori estimates. Using an $H^{0,1}$ energy estimate, we have 
\begin{align*}
&\frac{1}{2}\frac{d}{dt}\|X^\phi\|^2_{\tilde{H}^{0,1}}+\|\partial_1X^\phi\|^2_{\tilde{H}^{0,1}}\\
=&-\langle B(X^\phi,u^0)+B(u^0, X^\phi), X^\phi\rangle_{\tilde{H}^{0,1}}+ \langle \sigma(t,u^0(t))\phi(t), X^\phi\rangle_{\tilde{H}^{0,1}}.
\end{align*}

The first two terms on the roght hand side can be dealt by the same calculation as in the proof of Lemma \ref{wellposedness limit eq. of CLT}. For the third term we have
\begin{align*}
|\langle \sigma(t,u^0(t))\phi(t), X^\phi\rangle_{\tilde{H}^{0,1}}|\leqslant& \|\sigma(t,u^0)\|_{L_2(l^2, \tilde{H}^{0,1})}\|\phi(t)\|_{l^2}\|X^\phi(t)\|_{\tilde{H}^{0,1}}\\
\leqslant&  \tilde{K}_0+\tilde{K}_1\|u\|^2_{\tilde{H}^{0,1}}+\tilde{K}_2(\|\partial_1 u\|_H^2+\|\partial_1\partial_2u\|^2_H)+ C\|\phi\|^2_{l^2}\|X^\phi\|^2_{\tilde{H}^{0,1}}\\
\leqslant& C+C\|\phi\|^2_{l^2}\|X^\phi\|^2_{\tilde{H}^{0,1}},
\end{align*}
where we used (A2) in the second line.
Thus  we deduce that 
\begin{align*}
&\|X^\phi(t)\|^2_{H^{0,1}}+\int^t_0\|X^\phi(s)\|^2_{H^{1,1}}ds\\
\leqslant &C+C\int^t_0 \left(1+\|u^0\|^2_{H^{1,2}}+\|\phi\|^2_{l^2}\right)\|X^\phi\|^2_{H^{0,1}}ds.
\end{align*}

By Gronwall's inequality we have
\begin{align*}
&\|X^\phi(t)\|^2_{H^{0,1}}+\int^t_0\|X^\phi(s)\|^2_{H^{1,1}}ds\\
\leqslant &C\exp\left(\int^t_0 \left(1+\|u^0\|^2_{H^{1,2}}+\|\phi\|^2_{l^2}\right)ds\right)\leqslant C,
\end{align*}
where we used Lemma \ref{wellposedness result for u0 in H02}. 

The existence results will be given by  compactness arguments (see \cite[Theorem 3.1]{LZZ18}).  We put them in the following for the use  in the proof of next lemma.

Consider the  approximate equation:
\begin{equation}\label{approxi. eq. MDP}
\left\{
             \begin{aligned}
             &dX_\epsilon^\phi(t)=\partial^2_1 X_\epsilon^\phi(t)dt+\epsilon^2\partial_2^2X_\epsilon^\phi(t)dt-B(X_\epsilon^\phi,u^0)dt-B(u^0, X_\epsilon^\phi)dt+\sigma(t, u^0(t))\phi(t)dt,\\
             &X_\epsilon^\phi(0)=0.
             \end{aligned}
\right.
\end{equation}

It follows from classical theory on Navier-Stokes system that (\ref{approxi. eq. MDP}) has a unique global smooth solution $z^\phi_\epsilon$ for any fixed $\epsilon$.  Furthermore,  we have
\begin{align*}
\|X_\epsilon^\phi(t)\|^2_{H^{0,1}}+\int^t_0\|X_\epsilon^\phi(s)\|^2_{H^{1,1}}ds
\leqslant  C.
\end{align*}

Then  $\{X_\epsilon^\phi\}_{\epsilon>0}$ is uniformly bounded in  $L^\infty([0,T], \tilde{H}^{0,1})\bigcap L^2([0,T], \tilde{H}^{1,1})$, hence bounded in $L^4([0,T],H^{\frac{1}{2}})$ (by interpolation) and $L^4([0,T], L^4(\mathbb{T}^2))$ (by Sobolev embedding). Thus $B(X^\phi_\epsilon, u^0)$ and $B(u^0, X^\phi_\epsilon)$ are uniformly bounded in $L^2([0,T],H^{-1})$. Let $p\in(1,\frac{4}{3})$, we have
\begin{align*}
\int^T_0\|\sigma(s,u^0(s))\phi(s)\|^p_{H^{-1}}ds\leqslant &\int^T_0\|\sigma(s,u^0(s))\|^p_{L_2(l^2,H^{-1})}\|\phi(s)\|^p_{l^2}ds\\
\leqslant &C\int^T_0(1+\|\sigma(s,u^0(s))\|^4_{L_2(l^2,H^{-1})}+\|\phi(s)\|^2_{l^2})ds\\
\leqslant&C\int^T_0(1+\|u^0(s))\|^4_{H}+\|\phi(s)\|^2_{l^2})ds<\infty,
\end{align*}
where we used Young's inequality in the second line and (A0) in the third line.
It comes out that 
\begin{equation}\label{bded in Lp(H-1) MDP}
\{\partial_t X_\epsilon^\phi\}_{\epsilon>0}\text{  is uniformly bounded in } L^p([0,T],H^{-1}).\end{equation}
 Thus by  Aubin-Lions lemma (see \cite[Lemma 3.6]{LZZ18}),  there exists a $X^\phi\in L^2([0,T], {H})$ such that 
$$X_\epsilon^\phi\rightarrow X^\phi \text{ strongly in }L^2([0,T],H)\text{ as }\epsilon\rightarrow 0 \text{ (in the sense of subsequence)}.$$

Since $\{X_\epsilon^\phi\}_{\epsilon>0}$ is uniformly bounded in  $L^\infty([0,T], \tilde{H}^{0,1})\bigcap L^2([0,T], \tilde{H}^{1,1})$, there exists a $\tilde{X}\in L^\infty([0,T], \tilde{H}^{0,1})\bigcap L^2([0,T], \tilde{H}^{1,1})$ such that
$$X_\epsilon^\phi\rightarrow \tilde{X} \text{ weakly in }L^2([0,T],\tilde{H}^{1,1})\text{ as }\epsilon\rightarrow 0 \text{ (in the sense of subsequence)}.$$
$$X_\epsilon^\phi\rightarrow \tilde{X} \text{ weakly star in }L^\infty([0,T],\tilde{H}^{0,1})\text{ as }\epsilon\rightarrow 0 \text{ (in the sense of subsequence)}.$$
By the uniqueness of weak convergence limit, we deduce that $X^\phi=\tilde{X}$. 
By (\ref{bded in Lp(H-1) MDP}) and \cite[Theorem 2.2]{FG95}, we also have  for any $\delta>0$
$$X_\epsilon^\phi\rightarrow X^\phi \text{ strongly in }C([0,T],H^{-1-\delta})\text{ as }\epsilon\rightarrow 0 \text{ (in the sense of subsequence)}.$$

Now we use the above convergence  to prove that $X^\phi$ is a solution to (\ref{skeleton eq. MDP}). Note that for any $\varphi\in C^\infty([0,T]\times \mathbb{T}^2)$ with $\text{div} \varphi=0$, for any $t\in[0,T]$, $z_\epsilon^{\phi}$  satisfies
\begin{equation}\label{eq of approximate solution MDP}\aligned
&\langle X^\phi_\epsilon(t), \varphi(t)\rangle \\
=&\int^{t}_0\langle X^\phi_\epsilon, \partial_t\varphi\rangle-\langle \partial_1X^\phi_\epsilon,\partial_1 \varphi\rangle-\epsilon^2\langle \partial_2X^\phi_\epsilon,\partial_2 \varphi\rangle+\langle- B(X^\phi_\epsilon, u^0)-B(u^0, X^\phi_\epsilon)+\sigma(s, u^0)\phi, \varphi\rangle ds.
\endaligned
\end{equation}

Let $\epsilon\rightarrow 0$ in (\ref{eq of approximate solution MDP}),  we have $X^\phi\in L^\infty([0,T], \tilde{H}^{0,1})\bigcap L^2([0,T], \tilde{H}^{1,1})$ and
$$\partial_t X^\phi=\partial_1^2X^\phi-B(X^\phi,u^0)-B(u^0,X^\phi)+\sigma(t,u^0(t))\phi.$$
Since the right hand side belongs to $L^p([0,T],H^{-1})$, we deduce that $$X^\phi\in L^\infty([0,T], \tilde{H}^{0,1})\bigcap L^2([0,T], \tilde{H}^{1,1})\bigcap C([0,T],H^{-1}).$$ 

The uniqueness part is exactly the same as in Lemma \ref{wellposedness limit eq. of CLT}.

\end{proof}

The following Lemma shows that $I$ is a good rate function. The proof follows essentially the same argument as in \cite[Proposition 4.5]{WZZ}.
\begin{lemma}\label{good rate function MDP}
Assume (A0)-(A2) hold. For all $N<\infty$, the set 
$$K_N=\left\{g^0\left(\int^\cdot_0 \phi(s)ds\right): \phi\in S_N\right\}$$
is a compact subset in $L^\infty([0,T], H)\bigcap L^2([0,T], \tilde{H}^{1,0})\bigcap C([0,T],H^{-1})$.
\end{lemma}

\begin{proof}
By definition, we have
$$K_N=\left\{X^\phi: \phi\in L^2([0,T], l^2), \text{ }\int^T_0\|\phi(s)\|^2_{l^2}ds\leqslant N\right\}.$$

Let $\{X^{\phi_n}\}$ be a sequence in $K_N$ where $\{\phi_n\}\subset S_N$. Note that  $X^{\phi_n}$ is uniformly bounded in $L^\infty([0,T], H^{1,0})\cap L^2([0,T], H^{1,1})$. Thus  by weak compactness of $S_N$,  a similar argument as in the proof of Lemma \ref{wellposedness for skeleton eq. MDP} shows that there exists $\phi\in\mathcal{S}_N$ and $X'\in L^2([0,T],H)$ such that  the following convergence hold as $n\rightarrow \infty$ (in the sense of subsequence):  

$\phi_n\rightarrow \phi$ in $\mathcal{S}_N$ weakly,

$X^{\phi_n}\rightarrow X'$ in $L^2([0,T], H^{1,0})$ weakly,

$X^{\phi_n}\rightarrow X'$ in $L^\infty([0,T], H)$ weak-star,

$X^{\phi_n}\rightarrow X'$ in $L^2([0,T], H)$ strongly.

$X^{\phi_n}\rightarrow X'$ in $C([0,T], H^{-1-\delta})$ strongly for any $\delta>0$.

Then for any $\varphi\in C^\infty([0,T]\times \mathbb{T}^2)$ with $\text{div} \varphi=0$ and for any $t\in[0,T]$, $X^{\phi_n}$  satisfies
\begin{equation}\label{eq in good rate function}\aligned
&\langle X^{\phi_n}(t), \varphi(t)\rangle=\langle u_0,\varphi(0)\rangle\\
&+\int^{t}_0\langle X^{\phi_n}, \partial_t\varphi\rangle-\langle \partial_1X^{\phi_n},\partial_1 \varphi\rangle+\langle- B(X^{\phi_n},u^0)-B(u^0, X^{\phi_n})+\sigma(s, u^0)\phi_n, \varphi\rangle ds.
\endaligned
\end{equation}

Let $n\rightarrow \infty$, we deduce  that  $X'$ is a solution to (\ref{skeleton eq. MDP}). By the uniqueness of solution, we deduce that $X'=X^\phi$.

Our goal is to prove $X^{\phi_n}\rightarrow X^\phi$  in $L^\infty([0,T], H)\bigcap L^2([0,T], \tilde{H}^{1,0})\bigcap C([0,T],H^{-1})$. 

Let $w^n=X^{\phi_n}-X^\phi$, by a direct calculation, we have
\begin{align*}
&\|w^n(t)\|^2_H+2\int^t_0\|\partial_1w^n(s)\|^2_Hds\\
=&-2\int^t_0\langle w^n(s), B(X^{\phi_n}(s)-X^\phi(s), u^0(s))\rangle ds\\
&-2\int^t_0\langle w^n(s), B(u^0(s), X^{\phi_n}(s)-X^\phi(s))\rangle ds\\
&+2\int^t_0\langle w^n(s), \sigma(s, u^0(s))(\phi_n(s)-\phi(s))\rangle ds\\
\leqslant &2\int^t_0 |b(w^n, u^0, w^n)(s)|ds+2\int^t_0|\langle w^n(s), \sigma(s, u^0(s))(\phi_n(s)-\phi(s))\rangle| ds\\
\leqslant & \int^t_0 \|\partial_1w^n(s)\|^2_H+C(1+\|u^0(s)\|^2_{\tilde{H}^{1,1}})\|w^n(s)\|^2_Hds\\
&+C\int^t_0\|w^n(s)\|_H\|\phi_n(s)-\phi(s)\|_{l^2}(1+\|u^0(s)\|^2_H+\|\partial_1 u^0(s)\|^2_H )^\frac{1}{2} ds,
\end{align*}
where we used Lemma \ref{b(u,v,w)} and (A1) in the last inequality.

Note that $\phi_n$, $\phi$ are in $\mathcal{S}_N$, we have
\begin{align*}
&\|w^n(t)\|^2_H+\int^t_0\|\partial_1w^n(s)\|^2_Hds\\
\leqslant & C\int^t_0(1+\|u^0(s)\|^2_{\tilde{H}^{1,1}})\|w^n(s)\|^2_Hds\\
&+C\left(\int^t_0\|w^n(s)\|^2_H(1+\|u^0(s)\|^2_H+\|\partial_1 u^0(s)\|^2_H ) ds\right)^\frac{1}{2}\left(\int^t_0\|\phi_n(s)-\phi(s)\|_{l^2}^2\right)^\frac{1}{2}\\
\leqslant & C\int^t_0(1+\|u^0(s)\|^2_{\tilde{H}^{1,1}})\|w^n(s)\|^2_Hds\\
&+C\sqrt{N}\left(\int^t_0\|w^n(s)\|^2_H(1+\|u^0(s)\|^2_H+\|\partial_1 u^0(s)\|^2_H ) ds\right)^\frac{1}{2}.
\end{align*}

For any $\epsilon>0$, let 
$$A_\epsilon:=\{s\in[0,T]; \|w^n(s)\|_H>\epsilon\}.$$
Since $X^{\phi_n}\rightarrow X^\phi$ in $L^2([0,T], H)$ strongly, we have
$$\int^T_0\|w^n(s)\|^2_Hds\rightarrow 0,\text{ as }n\rightarrow\infty$$
and $\lim_{n\rightarrow\infty}Leb(A_\epsilon)=0$, where $Leb(B)$ means the Lebesgue measure of $B\in\mathcal{B}(\mathbb{R})$.  Thus we have
\begin{align*}
&\int^T_0(1+\|u^0(s)\|^2_{\tilde{H}^{1,1}})\|w^n(s)\|^2_Hds\\
\leqslant&\left(\int_{ A_\epsilon}+\int_{[0,T]\setminus A_\epsilon}\right)(1+\|u^0(s)\|^2_{\tilde{H}^{1,1}})\|w^n(s)\|^2_Hds\\
\leqslant& C\epsilon+ 2\int_{A_\epsilon}(1+\|u^0(s)\|^2_{\tilde{H}^{1,1}})(\|X^{\phi_n}(s)\|^2_H+\|X^\phi(s)\|^2_H)ds\\
\leqslant&C\epsilon+ C\int_{A_\epsilon}(1+\|u^0(s)\|^2_{\tilde{H}^{1,1}})ds\\
\rightarrow & \text{ }C\epsilon\text{ as } n\rightarrow \infty,
\end{align*}
where we used Lemma \ref{estimates from LZZ18} in the last line. A  similar argument also implies that 
\begin{align*}
\int^T_0 (1+\|u^0(s)\|^2_H+\|\partial_1u^0(s)\|^2_H)\|w^n(s)\|^2_Hds\leqslant C\epsilon.
\end{align*}
Hence we have
\begin{align*}
\sup_{t\in[0,T]}\|w^n(t)\|^2_H+\int^T_0\|\partial_1w^n(s)\|^2_Hds\leqslant C\epsilon+C\sqrt{\epsilon} \text{ as }n\rightarrow \infty.
\end{align*}

Since $\epsilon$ is arbitrary, we obtain that 
$$X^{\phi^n}\rightarrow X^\phi\text{ strongly in }L^\infty([0,T], H)\bigcap L^2([0,T], \tilde{H}^{1,0})\bigcap C([0,T],H^{-1}).$$

\end{proof}

The next step is to check Hypothesis 1.  To this end, recall that $Z^\varepsilon=\frac{u^\varepsilon-u^0}{\sqrt{\varepsilon}\lambda(\varepsilon)}$, then
\begin{equation}\label{MDP eq.}\aligned
dZ^\varepsilon(t)&=\partial^2_1 Z^\varepsilon(t)dt-B(Z^\varepsilon(t), u^0(t)+\sqrt{\varepsilon}\lambda(\varepsilon)Z^\varepsilon(t))dt-B(u^0(t), Z^\varepsilon(t))dt\\
&+\lambda^{-1}(\varepsilon)\sigma(t, u^0(t)+\sqrt{\varepsilon}\lambda(\varepsilon)Z^\varepsilon(t))dW(t),
\endaligned
\end{equation}
with initial value $Z^\varepsilon(0)=0$. The uniqueness of solution to \eqref{MDP eq.} is very similar to that of \eqref{equation after projection}.  Then it follows from Yamada-Watanabe theorem (See \cite[Appendix E]{LR15}) that there exists a Borel-measurable function
 $$g^\varepsilon: C([0,T],U)\rightarrow L^\infty([0,T], H)\bigcap L^2([0,T], \tilde{H}^{1,0})\bigcap C([0,T],H^{-1})$$
such that $Z^\varepsilon=g^\varepsilon(W)$ a.s..   

Now consider the following equation:
\begin{equation}\label{eq. for weak convergence MDP}\aligned
dX^\varepsilon(t)&=\partial^2_1 X^\varepsilon(t)dt-B(X^\varepsilon(t), u^0(t)+\sqrt{\varepsilon}\lambda(\varepsilon)X^\varepsilon(t))dt-B(u^0(t), X^\varepsilon(t))dt\\
&+\sigma(t, u^0(t)+\sqrt{\varepsilon}\lambda(\varepsilon)X^\varepsilon(t))v^\varepsilon(t)dt+\lambda^{-1}(\varepsilon)\sigma(t, u^0(t)+\sqrt{\varepsilon}\lambda(\varepsilon)X^\varepsilon(t))dW(t),\\
 X^\varepsilon(0)&=0,
\endaligned
\end{equation}
where $v^\varepsilon\in\mathcal{A}_N$ for some $N<\infty$.
Here $X^\varepsilon$ should have been denoted $X^\varepsilon_{v^\varepsilon}$ and the slight abuse of notation is for simplicity. 

\begin{lemma}\label{Girsanov thm to prove existence MDP}
Assume (A0)-(A3) hold with $K_2<\frac{2}{21}, \tilde{K}_2<\frac{1}{5}, L_2<\frac{1}{5}$ and $v^\varepsilon\in\mathcal{A}_N$ for some $N<\infty$. Then $X^\varepsilon=g^\varepsilon\left(W(\cdot)+\lambda(\varepsilon)\int^\cdot_0v^\varepsilon(s)ds\right)$ is the unique strong solution to (\ref{eq. for weak convergence MDP}).
\end{lemma}
\begin{proof}
Since $v^\varepsilon\in \mathcal{A}_N$,  by the Girsanov theorem (see \cite[Appendix I]{LR15}), $\tilde{W}(\cdot):=W(\cdot)+\lambda(\varepsilon)\int^\cdot_0v^\varepsilon(s)ds$ is an $l^2$-cylindrical Wiener-process under the probability measure
$$d\tilde{P}:=\exp\left\{-\lambda(\varepsilon)\int^T_0v^\varepsilon(s)dW(s)-\frac{1}{2}\lambda^2(\varepsilon)\int^T_0\|v^\varepsilon(s)\|^2_{l^2}ds\right\}dP.$$
Then $(X^\varepsilon, \tilde{W})$ is the solution to (\ref{MDP eq.}) on the stochastic basis $(\Omega, \mathcal{F}, \tilde{P})$. 
Thus $(X^\varepsilon, W)$ satisfies the condition of the definition of weak solution (see \cite[Definition 4.1]{LZZ18}) and hence is a weak solution to (\ref{eq. for weak convergence MDP}) on the stochastic basis $(\Omega, \mathcal{F}, {P})$ and $X^\varepsilon=g^\varepsilon\left(W(\cdot)+\lambda(\varepsilon)\int^\cdot_0v^\varepsilon(s)ds\right)$.

If $\tilde{X^\varepsilon}$ and $X^\varepsilon$ are two weak solutions to  (\ref{eq. for weak convergence MDP}) on the same stochastic basis $(\Omega, \mathcal{F}, {P})$.  Let $W^\varepsilon=X^\varepsilon-\tilde{X^\varepsilon}$ and $q(t)=k\int^t_0(\|u^0+\sqrt{\varepsilon}\lambda(\varepsilon)X^\varepsilon(s)\|^2_{\tilde{H}^{1,1}}+\|v^\varepsilon(s)\|^2_{l^2})ds$ for some constant $k$. Applying It\^o's formula to $e^{-q(t)}\|W^\varepsilon(t)\|^2_H$, we have 
\begin{align*}
&e^{-q(t)}\|W^\varepsilon(t)\|^2_H+2\int^t_0e^{-q(s)}\|\partial_1 W^\varepsilon(s)\|^2_Hds\\
=&-k\int^t_0e^{-q(s)}\|W^\varepsilon(s)\|^2_H(\|u^0+\sqrt{\varepsilon}\lambda(\varepsilon)X^\varepsilon(s)\|^2_{\tilde{H}^{1,1}}+\|v^\varepsilon(s)\|^2_{l^2})ds\\
&-2\int^t_0e^{-q(s)}b(W^\varepsilon, u^0+\sqrt{\varepsilon}\lambda(\varepsilon)X^\varepsilon, W^\varepsilon)ds\\
&+2\int^t_0e^{-q(s)}\langle \sigma(s, u^0+\sqrt{\varepsilon}\lambda(\varepsilon) X^\varepsilon)v^\varepsilon-\sigma(s, u^0+\sqrt{\varepsilon}\lambda(\varepsilon)\tilde{X}^\varepsilon)v^\varepsilon, W^\varepsilon(s)\rangle ds\\
&+2\lambda^{-1}(\varepsilon)\int^t_0e^{-q(s)}\langle W^\varepsilon(s), (\sigma(s,u^0+\sqrt{\varepsilon}\lambda(\varepsilon)X^\varepsilon)-\sigma(s, u^0+\sqrt{\varepsilon}\lambda(\varepsilon)\tilde{X}^\varepsilon))dW(s)\rangle\\
&+\lambda^{-2}(\varepsilon)\int^t_0e^{-q(s)}\|\sigma(s,u^0+\sqrt{\varepsilon}\lambda(\varepsilon)X^\varepsilon)-\sigma(s, u^0+\sqrt{\varepsilon}\lambda(\varepsilon)\tilde{X}^\varepsilon)\|^2_{L_2(l^2,H)}ds.
\end{align*}
By Lemma \ref{b(u,v,w)}, there exists constants $\tilde{\alpha}\in(0,1)$ and $\tilde{C}$ such that 
$$|b(W^\varepsilon, u^0+\sqrt{\varepsilon}\lambda(\varepsilon)X^\varepsilon, W^\varepsilon)|\leqslant \tilde{\alpha}\|\partial_1 W^\varepsilon\|^2_H+\tilde{C}(1+\|u^0+\sqrt{\varepsilon}\lambda(\varepsilon)X^\varepsilon\|^2_{\tilde{H}^{1,1}})\|W^\varepsilon\|^2_H.$$
We also have 
\begin{align*}
&2|\langle \sigma(s,u^0+\sqrt{\varepsilon}\lambda(\varepsilon)X^\varepsilon)v^\varepsilon-\sigma(s, u^0+\sqrt{\varepsilon}\lambda(\varepsilon)\tilde{X}^\varepsilon)v^\varepsilon, W^\varepsilon\rangle|\\
\leqslant &2\|(\sigma(s,u^0+\sqrt{\varepsilon}\lambda(\varepsilon)X^\varepsilon)-\sigma(s, u^0+\sqrt{\varepsilon}\lambda(\varepsilon)\tilde{X}^\varepsilon))v^\varepsilon\|_H\|W^\varepsilon\|_H\\
\leqslant &\|\sigma(s,u^0+\sqrt{\varepsilon}\lambda(\varepsilon)X^\varepsilon)-\sigma(s, u^0+\sqrt{\varepsilon}\lambda(\varepsilon)\tilde{X}^\varepsilon)\|^2_{L_2(l^2,H)}+\|v^\varepsilon\|^2_{l^2}\|W^\varepsilon\|^2_H.
\end{align*}

By (A3), we have
\begin{align*}
&\|\sigma(s,u^0+\sqrt{\varepsilon}\lambda(\varepsilon)X^\varepsilon)-\sigma(s, u^0+\sqrt{\varepsilon}\lambda(\varepsilon)\tilde{X}^\varepsilon)\|^2_{L_2(l^2,H)}\\
\leqslant & \sqrt{\varepsilon}\lambda(\varepsilon) (L_1\|W^\varepsilon\|^2_H+L_2\|\partial_1 W^\varepsilon\|^2_{{H}}).
\end{align*}

By the Burkh\"older-Davis-Gundy's inequality (see \cite[Appendix D]{LR15}), we have 
\begin{align*}
&2\lambda^{-1}(\varepsilon)|E[\sup_{r\in[0,t]}\int^r_0e^{-q(s)}\langle W^\varepsilon(s), (\sigma(s,u^0+\sqrt{\varepsilon}\lambda(\varepsilon)X^\varepsilon)-\sigma(s, u^0+\sqrt{\varepsilon}\lambda(\varepsilon)\tilde{X}^\varepsilon))dW(s)\rangle]|\\
\leqslant &6\lambda^{-1}(\varepsilon) E\left(\int^t_0e^{-2q(s)}\|\sigma(s,u^0+\sqrt{\varepsilon}\lambda(\varepsilon)X^\varepsilon)-\sigma(s, u^0+\sqrt{\varepsilon}\lambda(\varepsilon)\tilde{X}^\varepsilon)\|^2_{L_2(l^2,H)}\|W^\varepsilon(s)\|^2_Hds\right)^\frac{1}{2}\\
\leqslant& \sqrt{\varepsilon}E(\sup_{s\in[0,t]}(e^{-q(s)}\|W^\varepsilon(s)\|^2_H))+9\sqrt{\varepsilon} E\int^t_0e^{-q(s)}(L_1\|W^\varepsilon(s)\|^2_H+L_2\|\partial_1 W^\varepsilon(s)\|^2_H)ds,
\end{align*}
where we used (A3).

Let $k>2\tilde{C}$ and we may  assume $\sqrt{\varepsilon}\lambda(\varepsilon)<1$,  by (A3)  we have 
\begin{align*}
&e^{-q(t)}\|W^\varepsilon(t)\|^2_H+(2-2\tilde{\alpha}-L_2\varepsilon\lambda^2(\varepsilon))\int^t_0e^{-q(s)}\|\partial_1 W^\varepsilon(s)\|^2_Hds\\
\leqslant& C\int^t_0e^{-q(s)}\|W^\varepsilon(s)\|^2_Hds\\
&+2\lambda^{-1}(\varepsilon)\int^t_0e^{-q(s)}\langle W^\varepsilon(s),  (\sigma(s,u^0+\sqrt{\varepsilon}\lambda(\varepsilon)X^\varepsilon)-\sigma(s, u^0+\sqrt{\varepsilon}\lambda(\varepsilon)\tilde{X}^\varepsilon))dW(s)\rangle.
\end{align*}

Let $\varepsilon$ be small enough such that $1-\sqrt{\varepsilon}-L_2\varepsilon\lambda^2(\varepsilon)-9\sqrt{\varepsilon}L_2>0$. Then we have
$$E(\sup_{s\in[0,t]}(e^{-q(s)}\|W^\varepsilon(s)\|^2_H))\leqslant CE\int^t_0e^{-q(s)}\|W^\varepsilon(s)\|^2_Hds.$$
By Gronwall's inequality we obtain $W^\varepsilon=0$ $P$-a.s., i.e. $\tilde{X^\varepsilon}=X^\varepsilon$ $P$-a.s..

 Then by the Yamada-Watanabe theorem,  we have $X^\varepsilon$ is the unique strong solution to (\ref{eq. for weak convergence MDP}).

\end{proof}

\begin{lemma}\label{estimate eq. for weak convergence}
Assume $X^\varepsilon$ is a solution to (\ref{eq. for weak convergence MDP}) with $v^\varepsilon\in\mathcal{A}_N$ and $\varepsilon<1$ small enough. Then we have
\begin{equation}\label{eq01 in lemma estimate eq. for weak convergence MDP}
E(\sup_{t\in[0,T]}\|X^\varepsilon(t)\|^4_H)+E\int^T_0(\|X^\varepsilon(s)\|^2_H+1)\|X^\varepsilon(s)\|^2_{\tilde{H}^{1,0}}ds\leqslant C(N).
\end{equation}

Moreover, there exists $k>0$ such that 
\begin{equation}\label{eq02 in lemma estimate eq. for weak convergence MDP}
E(\sup_{t\in[0,T]}e^{-kg(t)}\|X^\varepsilon(t)\|^2_{\tilde{H}^{0,1}})+E\int^{T}_0e^{-kg(s)}\|X^\varepsilon(s)\|^2_{\tilde{H}^{1,1}}ds\leqslant C(N),
\end{equation}
where $g(t)=\int^t_0\|\partial_1 X^\varepsilon(s)\|^2_{H}ds$ and $C(N)$ is a constant depend on $N$ but independent of $\varepsilon$. 
\end{lemma}

\begin{proof}
We prove (\ref{eq01 in lemma estimate eq. for weak convergence MDP}) by two steps of estimates. 
For the first step, applying It\^o's formula to $\|X^\varepsilon(t)\|^2_H$, we have
\begin{align*}
&\|X^\varepsilon(t)\|^2_H+2\int^t_0\|\partial_1 X^\varepsilon(s)\|^2_Hds\\
=&-2\int^t_0b(X^\varepsilon, u^0, X^\varepsilon)ds+2\int^t_0 \langle X^\varepsilon(s), \sigma(s,u^0+\sqrt{\varepsilon}\lambda(\varepsilon) X^\varepsilon(s))v^\varepsilon(s)\rangle ds\\
&+2\lambda^{-1}(\varepsilon)\int^t_0\langle X^\varepsilon(s), \sigma(s,u^0+\sqrt{\varepsilon}\lambda(\varepsilon)X^\varepsilon(s))dW(s)\rangle\\
&+\lambda^{-2}(\varepsilon)\int^t_0\|\sigma(s,u^0+\sqrt{\varepsilon}\lambda(\varepsilon)X^\varepsilon(s))\|^2_{L_2(l^2, H)}ds\\
\leqslant& \int^t_0 (\frac{1}{2}\|\partial_1X^\varepsilon(s)\|^2_H+C(1+\|u^0\|^2_{H^{1,1}})\|X^\varepsilon\|^2_H)ds\\
&+ \int^t_0(\|X^\varepsilon(s)\|^2_H\|v^\varepsilon(s)\|^2_{l^2}+\|\sigma(s, u^0+\sqrt{\varepsilon}\lambda(\varepsilon)X^\varepsilon(s))\|^2_{L_2(l^2,H)})ds\\
&+2\lambda^{-1}(\varepsilon)\int^t_0\langle X^\varepsilon(s), \sigma(s,u^0+\sqrt{\varepsilon}\lambda(\varepsilon)X^\varepsilon(s))dW(s)\rangle\\
&+\lambda^{-2}(\varepsilon)\int^t_0\|\sigma(s,u^0+\sqrt{\varepsilon}\lambda(\varepsilon)X^\varepsilon(s))\|^2_{L_2(l^2, H)}ds\\
\leqslant &\int^t_0 (\frac{1}{2}\|\partial_1X^\varepsilon(s)\|^2_H+C(1+\|u^0\|^2_{H^{1,1}})\|X^\varepsilon\|^2_H)ds+\int^t_0\|X^\varepsilon(s)\|^2_H\|v^\varepsilon(s)\|^2_{l^2}ds\\
&+ (1+\lambda^{-2}(\varepsilon))\int^t_0(K_0+K_1\|u^0+\sqrt{\varepsilon}\lambda(\varepsilon)X^\varepsilon\|^2_H+K_2\|\partial_1(u^0+\sqrt{\varepsilon}\lambda(\varepsilon)X^\varepsilon)\|_H^2)ds\\
&+2\lambda^{-1}(\varepsilon)\int^t_0\langle X^\varepsilon(s), \sigma(s,u^0+\sqrt{\varepsilon}\lambda(\varepsilon)X^\varepsilon(s))dW(s)\rangle,
\end{align*}
where we used (A1) in the last inequality.

Note that $v^\varepsilon\in \mathcal{A}_N$, by Lemma \ref{estimates from LZZ18} and Gronwall's inequality, 
\begin{align*}
&\|X^\varepsilon(t)\|^2_H+(\frac{3}{2}-\varepsilon K_2-\lambda^2(\varepsilon)\varepsilon K_2)\int^t_0\|\partial_1 X^\varepsilon(s)\|^2_Hds\\
\leqslant &(C+2\lambda^{-1}(\varepsilon)\int^t_0\langle X^\varepsilon(s), \sigma(s,u^0+\sqrt{\varepsilon}\lambda(\varepsilon)X^\varepsilon(s))dW(s)\rangle)e^{C_1(N)}.
\end{align*}

For the term on the right hand side, by the Burkh\"older-Davis-Gundy inequality we have

\begin{align*}
&2\lambda^{-1}(\varepsilon) e^{C_1(N)} E\left(\sup_{0\leqslant s\leqslant t}|\int^s_0\langle X^\varepsilon(r), \sigma(r,u^0+\sqrt{\varepsilon}\lambda(\varepsilon)X^\varepsilon(r))dW(r)\rangle|\right)\\
\leqslant &6\lambda^{-1}(\varepsilon) e^{C_1(N)}E\left(\int^t_0\|X^\varepsilon(r)\|^2_H\|\sigma(r,u^0+\sqrt{\varepsilon}\lambda(\varepsilon)X^\varepsilon(r))\|^2_{L_2(l^2,H)}ds\right)^\frac{1}{2}\\
\leqslant &\lambda^{-1}(\varepsilon) E[\sup_{0\leqslant s\leqslant t}(\|X^\varepsilon(s)\|^2_H)]\\
&+9\lambda^{-2}(\varepsilon) e^{C_1(N)}E\int^t_0[K_0+K_1\|u^0+\sqrt{\varepsilon}\lambda(\varepsilon)X^\varepsilon(s)\|^2_H+K_2\|\partial_1(u^0+\sqrt{\varepsilon}\lambda(\varepsilon)X^\varepsilon(s))\|^2_H]ds,
\end{align*}
where  $(9{\varepsilon}e^{C_1(N)}+\varepsilon \lambda^2(\varepsilon)+\varepsilon)K_2-\frac{3}{4}<0$ (this can be done since $\sqrt{\varepsilon}\lambda(\varepsilon)\rightarrow 0$) and we used (A1) in the last inequality. Thus we have

\begin{align*}
&E[\sup_{s\in[0,t]}(\|X^\varepsilon(t)\|^2_H)]+E\int^t_0\|\partial_1 X^\varepsilon(s)\|^2_Hds\\
\leqslant&C(N)+C(N)\int^t_0E[\sup_{r\in[0,s]}(\|X^\varepsilon(r)\|^2_H)]ds.
\end{align*}

Then by  Gronwall's inequality  we have
\begin{equation}\label{eq1 in lemma estimate eq. for weak convergence}\aligned
E(\sup_{0\leqslant t\leqslant T}\|X^\varepsilon(t)\|^2_H)+E\int^T_0\|\partial_1 X^\varepsilon(s)\|^2_{H}ds\leqslant C(N).
\endaligned
\end{equation}

Now by It\^o's formula we have
\begin{equation}\label{L4 estimate step1}\aligned
\|X^\varepsilon(t)\|^4_H=&-4\int^t_0\|X^\varepsilon(s)\|^2_H\|\partial_1X^\varepsilon(s)\|^2_Hds-4\int^t_0\|X^\varepsilon(s)\|^2_Hb(X^\varepsilon, u^0, X^\varepsilon)ds\\
&+4\int^t_0\|X^\varepsilon(s)\|_H^2\langle\sigma(s,u^0+\sqrt{\varepsilon}\lambda(\varepsilon)X^\varepsilon(s))v^\varepsilon(s),X^\varepsilon(s)\rangle ds\\
&+2\lambda^{-2}(\varepsilon)\int^t_0\|X^\varepsilon(s)\|^2_H\|\sigma(s,u^0+\sqrt{\varepsilon}\lambda(\varepsilon)X^\varepsilon(s))\|^2_{L_2(l^2,H)}ds\\
&+4\lambda^{-2}(\varepsilon)\int^t_0\|\sigma(s,u^0+\sqrt{\varepsilon}\lambda(\varepsilon)X^\varepsilon(s))^*(X^\varepsilon)\|^2_{l^2}ds\\
&+4\lambda^{-1}(\varepsilon)\int^t_0\|X^\varepsilon(s)\|^2_H\langle X^\varepsilon(s), \sigma(s,u^0+\sqrt{\varepsilon}\lambda(\varepsilon)X^\varepsilon(s))dW(s)\rangle_H\\
=:&-4\int^t_0\|X^\varepsilon\|^2_H\|\partial_1X^\varepsilon(s)\|^2_Hds+I_0+I_1+I_2+I_3+I_4.
\endaligned
\end{equation}
By Lemma \ref{b(u,v,w)}, 
\begin{align*}
|I_0(t)|\leqslant 4\int^t_0\|X^\varepsilon\|^2_H(\frac{1}{4}\|\partial_1X^\varepsilon\|^2_H+C(1+\|u^0\|^2_{H^{1,1}})\|X^\varepsilon\|^2_H))ds.
\end{align*}
By (A1) we have 
\begin{align*}
I_1(t)\leqslant& 4\int^t_0\|X^\varepsilon(s)\|^2_H\|\sigma(s,u^0+\sqrt{\varepsilon}\lambda(\varepsilon)X^\varepsilon(s))\|_{L_2(l^2,H)}\|v^\varepsilon(s)\|_{l^2}\|X^\varepsilon(s)\|_Hds\\
\leqslant& 2\int^t_0\|X^\varepsilon(s)\|^2_H(K_0+K_1\|u^0+\sqrt{\varepsilon}\lambda(\varepsilon)X^\varepsilon(s)\|^2_H\\
&+K_2\|\partial_1(u^0+\sqrt{\varepsilon}\lambda(\varepsilon)X^\varepsilon(s))\|^2_H+\|v^\varepsilon(s)\|^2_{l^2}\|X^\varepsilon(s)\|^2_H)ds,
\end{align*}
and 
\begin{align*}
I_2+I_3\leqslant &6\lambda^{-2}(\varepsilon)\int^t_0\|\sigma(s,u^0+\sqrt{\varepsilon}\lambda(\varepsilon)X^\varepsilon(s))\|_{L_2(l^2,H)}^2\|X^\varepsilon(s)\|_H^2ds\\
\leqslant& 6\lambda^{-2}(\varepsilon)\int^t_0(K_0+K_1\|u^0+\sqrt{\varepsilon}\lambda(\varepsilon)X^\varepsilon(s)\|^2_H\\
&+K_2\|\partial_1(u^0+\sqrt{\varepsilon}\lambda(\varepsilon)X^\varepsilon(s))\|^2_H)\|X^\varepsilon(s)\|^2_Hds.
\end{align*}
Thus we have
\begin{align*}
&\|X^\varepsilon(t)\|^4_H+(3-2\varepsilon\lambda^2(\varepsilon) K_2-6\varepsilon K_2)\int^t_0\|X^\varepsilon(s)\|^2_H\|\partial_1X^\varepsilon(s)\|^2_Hds\\
\leqslant& I_4+C+C\int^t_0(1+\|u^0(s)\|^2_{H^{1,1}}+\|v^\varepsilon(s)\|^2_{l^2})\|X^\varepsilon(s)\|^4_H)ds.
\end{align*}
Since $v^\varepsilon\in\mathcal{A}_N$, by Gronwall's inequality we have
\begin{align*}
&\|X^\varepsilon(t)\|^4_H+(3-2\varepsilon\lambda^2(\varepsilon) K_2-6\varepsilon K_2)\int^t_0\|X^\varepsilon(s)\|^2_H\|\partial_1X^\varepsilon(s)\|^2_Hds\\
\leqslant &\left(I_4+C\right)e^{C_2(N)}.
\end{align*}

Then the Burkh\"older-Davis-Gundy inequality, the Young's inequality and (A1) imply that
\begin{align*}
E(\sup_{s\in[0,t]}I_4(s))\leqslant& 12\lambda^{-1}(\varepsilon) E\left(\int^t_0\|\sigma(s,u^0+\sqrt{\varepsilon}\lambda(\varepsilon)X^\varepsilon(s))\|^2_{L_2(l^2,H)}\|X^\varepsilon(s)\|^6_Hds\right)^\frac{1}{2}\\
\leqslant& \lambda^{-1}(\varepsilon) E(\sup_{s\in[0,t]}\|X^\varepsilon(s)\|^4_H)+36\lambda^{-1}(\varepsilon) E\int^t_0(K_0+K_1\|u^0+\sqrt{\varepsilon}\lambda(\varepsilon)X^\varepsilon(s)\|^2_H\\
&+K_2\|\partial_1(u^0+\sqrt{\varepsilon}\lambda(\varepsilon)X^\varepsilon(s))\|^2_H)\|X^\varepsilon(s)\|^2_Hds.
\end{align*}
Let $\varepsilon$ small enough such that $3-2\varepsilon\lambda^2(\varepsilon) K_2-6\varepsilon K_2-36\varepsilon K_2e^{C_2(N)}>0$ and $\lambda^{-1}(\varepsilon)e^{C_2(N)}<1$. Then the above estimates and (\ref{eq01 in lemma estimate eq. for weak convergence MDP}) imply that
\begin{align*}
&E(\sup_{s\in[0,t]}\|X^\varepsilon(s)\|^4_H)+\int^t_0 \|X^\varepsilon(s)\|^2_H\|X^\varepsilon(s)\|^2_{\tilde{H}^{1,0}}ds\\
\leqslant &C(N)+C(N)E\int^t_0\|X^\varepsilon(s)\|^4_Hds,
\end{align*}
which by Gronwall's inequality yields that 
\begin{align*}
E(\sup_{s\in[0,t]}\|X^\varepsilon(s)\|^4_H)+\int^t_0 \|X^\varepsilon(s)\|^2_H\|X^\varepsilon(s)\|^2_{\tilde{H}^{1,0}}ds\leqslant C(N).
\end{align*}

For (\ref{eq02 in lemma estimate eq. for weak convergence MDP}), let $h(t)=kg(t)+\int^t_0\|v^\varepsilon(s)\|^2_{l^2}ds$ for some universal constant $k$. Applying It\^o's formula to $e^{-h(t)}\|X^\varepsilon(t)\|^2_{\tilde{H}^{0,1}}$(by applying It\^o's formula to its finite-dimension projection first and then passing to the limit), we have
\begin{align*}
&e^{-h(t)}\|X^\varepsilon(t)\|^2_{\tilde{H}^{0,1}}+2\int^t_0e^{-h(s)}(\|\partial_1 X^\varepsilon(s)\|^2_H+\|\partial_1\partial_2 X^\varepsilon(s)\|^2_H)ds\\
=&-\int^t_0e^{-h(s)}(k\|\partial_1X^\varepsilon(s)\|^2_H+\|v^\varepsilon(s)\|^2_{l^2})\|X^\varepsilon(s)\|^2_{\tilde{H}^{0,1}}ds\\
&-2\int^t_0e^{-h(s)} b(X^\varepsilon,u^0,X^\varepsilon)ds-2\int^t_0e^{-h(s)}\langle \partial_2 X^\varepsilon(s), \partial_2(X^\varepsilon\cdot \nabla (u^0+\sqrt{\varepsilon}\lambda(\varepsilon) X^\varepsilon))(s)\rangle ds\\
&-2\int^t_0e^{-h(s)}\langle \partial_2 X^\varepsilon(s), \partial_2(u^0\cdot \nabla X^\varepsilon)(s)\rangle ds\\
&+2\int^t_0e^{-h(s)}\langle X^\varepsilon(s), \sigma(s,u^0+\sqrt{\varepsilon}\lambda(\varepsilon) X^\varepsilon(s))v^\varepsilon(s)\rangle_{\tilde{H}^{0,1}}ds\\
&+2\lambda^{-1}(\varepsilon)\int^t_0 e^{-h(s)}\langle X^\varepsilon(s), \sigma(s,u^0+\sqrt{\varepsilon}\lambda(\varepsilon) X^\varepsilon(s))dW(s)\rangle_{\tilde{H}^{0,1}}\\
&+\lambda^{-2}(\varepsilon)\int^t_0e^{-h(s)}\|\sigma(s,u^0+\sqrt{\varepsilon}\lambda(\varepsilon) X^\varepsilon(s))\|^2_{L_2(l^2, \tilde{H}^{0,1})}ds.
\end{align*}
By Lemma \ref{b(u,v,w)}, we have
$$2|b(X^\varepsilon,u^0,X^\varepsilon)|\leqslant \alpha\|\partial_1X^\varepsilon\|^2_H+C(1+\|u^0\|^2_{\tilde{H}^{1,1}})\|X^\varepsilon\|^2_H,$$
where $\alpha<\frac{1}{3}$.
By Lemma \ref{estimate for b with partial_2},  there exists $C_1$,
$$2\sqrt{\varepsilon}\lambda(\varepsilon)|\langle \partial_2 X^\varepsilon, \partial_2(X^\varepsilon\cdot \nabla X^\varepsilon)\rangle|\leqslant \alpha\|\partial_1\partial_2X^\varepsilon\|^2_H+C_1(1+\|\partial_1X^\varepsilon\|^2_H)\|\partial_2X^\varepsilon\|^2_H.$$
By Lemma \ref{b(u,v,w)}, we have
\begin{align*}
2|\langle \partial_2 X^\varepsilon, \partial_2(X^\varepsilon\cdot \nabla u^0))\rangle |\leqslant& 2|b(\partial_2X^\varepsilon, u^0,\partial_2 X^\varepsilon)|+2|b(X^\varepsilon, \partial_2u^0, \partial_2X^\varepsilon)|\\
\leqslant& \alpha(\|X^\varepsilon\|^2_{\tilde{H}^{1,0}}+\|\partial_2X^\varepsilon\|^2_{\tilde{H}^{1,0}})+C\|u^0\|^2_{\tilde{H}^{1,2}}\|\partial_2X^\varepsilon\|^2_H.
\end{align*}
Similarly, 
$$|\langle \partial_2 X^\varepsilon(s), \partial_2(u^0\cdot \nabla X^\varepsilon)(s)\rangle|=|b(\partial_2u^0,X^\varepsilon,\partial_2 X^\varepsilon)|\leqslant \alpha\|X^\varepsilon\|^2_{\tilde{H}^{1,1}}+C\|u^0\|^2_{\tilde{H}^{1,1}}\|\partial_2X^\varepsilon\|^2_H. $$
By Young's inequality, 
$$2|\langle X^\varepsilon(s), \sigma(s,u^0+\sqrt{\varepsilon}\lambda(\varepsilon) X^\varepsilon)v^\varepsilon(s)\rangle_{\tilde{H}^{0,1}}|\leqslant \|X^\varepsilon\|^2_{\tilde{H}^{0,1}}\|v^\varepsilon\|^2_{l^2}+\|\sigma(s,u^0+\sqrt{\varepsilon}\lambda(\varepsilon) X^\varepsilon)\|^2_{L_2(l^2,\tilde{H}^{0,1})}.$$
Choosing $k>2C_1\sqrt{\varepsilon}\lambda(\varepsilon)$, we have
\begin{align*}
&e^{-h(t)}\|X^\varepsilon(t)\|^2_{\tilde{H}^{0,1}}+(2-3\alpha)\int^t_0e^{-h(s)}\| X^\varepsilon(s)\|^2_{\tilde{H}^{1,1}}ds\\
\leqslant&C\int^t_0e^{-h(s)}(1+\|u^0\|^2_{\tilde{H}^{1,2}})\|X^\varepsilon(s)\|^2_{\tilde{H}^{0,1}}ds\\
&+(1+\lambda^{-2}(\varepsilon))\int^t_0e^{-h(s)}\|\sigma(s,u^0+\sqrt{\varepsilon}\lambda(\varepsilon) X^\varepsilon)\|^2_{L_2(l^2, \tilde{H}^{0,1})}ds\\
&+2\lambda^{-1}(\varepsilon)\int^t_0 e^{-h(s)}\langle X^\varepsilon(s), \sigma(s,u^0+\sqrt{\varepsilon}\lambda(\varepsilon) X^\varepsilon)dW(s)\rangle_{\tilde{H}^{0,1}}.
\end{align*}
By (A2) we have
\begin{align*}
&(1+\lambda^{-2}(\varepsilon))\|\sigma(s,u^0+\sqrt{\varepsilon}\lambda(\varepsilon) X^\varepsilon)\|^2_{L_2(l^2, \tilde{H}^{0,1})}\leqslant  C(1+\|u^0\|^2_{\tilde{H}^{1,1}})\\
&+(1+\lambda^{-2}(\varepsilon))\left(\tilde{K}_0+\tilde{K}_1\varepsilon\lambda^2(\varepsilon)\|X^\varepsilon\|^2_{\tilde{H}^{0,1}}+\tilde{K}_2\varepsilon\lambda^2(\varepsilon)(\|\partial_1X^\varepsilon\|^2_H+\|\partial_1\partial_2X^\varepsilon\|^2_{H})\right).
\end{align*}
By the Burkh\"older-Davis-Gundy inequality we have
\begin{align*}
&2\lambda^{-1}(\varepsilon) E\left(\sup_{s\in[0,t]}|\int^s_0e^{-h(r)}\langle X^\varepsilon(r), \sigma(r,u^0+\sqrt{\varepsilon}\lambda(\varepsilon) X^\varepsilon)dW(r)\rangle_{\tilde{H}^{0,1}}|\right)\\
\leqslant &6\lambda^{-1}(\varepsilon) E\left(\int^{t}_0e^{-2h(s)}\|X^\varepsilon(s)\|^2_{\tilde{H}^{0,1}}\|\sigma(s,u^0+\sqrt{\varepsilon}\lambda(\varepsilon) X^\varepsilon)\|^2_{L_2(l^2,\tilde{H}^{0,1})}ds\right)^\frac{1}{2}\\
\leqslant &\lambda^{-1}(\varepsilon) E[\sup_{s\in[0,t]}(e^{-h(s)}\|X^\varepsilon(s)\|^2_{\tilde{H}^{0,1}})]+\lambda^{-1}(\varepsilon)C\int^t_0e^{-h(s)}(1+\|u^0\|^2_{\tilde{H}^{1,1}})ds\\
&+9\varepsilon\lambda(\varepsilon) E\int^{t}_0e^{-h(s)}[\tilde{K_1}\|X^\varepsilon(s)\|^2_{\tilde{H}^{0,1}}+\tilde{K_2}(\|\partial_1 X^\varepsilon(s)\|^2_H+\|\partial_1\partial_2 X^\varepsilon(s)\|^2_H)]ds,
\end{align*}
where  we choose $\varepsilon$ small enough such that $(9\varepsilon\lambda(\varepsilon)+{\varepsilon}\lambda^2(\varepsilon)+\varepsilon)\tilde{K_2}<1-3\alpha$  and we used (A2) in the last inequality.

Combine the above estimates,  we have
\begin{align*}
&E(\sup_{s\in[0,{t}]}e^{-h(s)}\|X^\varepsilon(s)\|^2_{\tilde{H}^{0,1}})+E\int^{t}_0e^{-h(s)}\|X^\varepsilon(s)\|^2_{\tilde{H}^{1,1}}ds\\
\leqslant &C+CE\left(\int^{t}_0e^{-h(s)}(1+\|u^0(s)\|^2_{\tilde{H}^{1,2}})\|X^\varepsilon(s)\|^2_{\tilde{H}^{0,1}}ds\right)
\end{align*} 

Then Gronwall's inequality and \eqref{apriori estimate H02} imply  that 
\begin{align*}
E(\sup_{0\leqslant t\leqslant T}e^{-h(t)}\|X^\varepsilon(t)\|^2_{\tilde{H}^{0,1}})+E\int^{T}_0e^{-h(s)}\|X^\varepsilon(s)\|^2_{\tilde{H}^{1,1}}ds\leqslant C.
\end{align*}
Since $v^\varepsilon\in \mathcal{S}_N$, we deduce that
\begin{equation}\label{eq2 in lemma estimate eq. for weak convergence}
E(\sup_{t\in[0,T]}e^{-kg(t)}\|X^\varepsilon(t)\|^2_{\tilde{H}^{0,1}})+E\int^{T}_0e^{-kg(s)}\|X^\varepsilon(s)\|^2_{\tilde{H}^{1,1}}ds\leqslant C.
\end{equation}

\end{proof}

Similar as \cite[lemma 4.3]{LZZ18}, we have the following tightness lemma:
\begin{lemma}\label{tightness lemma}
Assume $X^\varepsilon$ is a solution to (\ref{eq. for weak convergence MDP}) with $v^\varepsilon\in\mathcal{A}_N$ and $\varepsilon$ small enough. 
There exists $\varepsilon_0>0$, such that
$\{X^\varepsilon\}_{\varepsilon\in(0,\varepsilon_0)}$ is tight in the space
$$\chi=C([0,T],H^{-1})\bigcap L^2([0,T],H)\bigcap L^2_w([0,T], H^{1,1})\bigcap L^\infty_{w^*}([0,T], H^{0,1}),$$
where $L^2_w$ denotes the weak topology and $L^\infty_{w^*}$ denotes the weak star topology.
\end{lemma}

\begin{proof}
Note  that the law of $Z^\varepsilon_v$ is defined on the path space $C([0,T],H^{-1})$. First we should point out that it can be restricted to $\chi$. We denote the space  $C([0,T], H^{-1})$ by $X$ with Borel $\sigma$-algebra $\mathcal{B}(X)$. 

For $N\in\mathbb{N}$, let 
$$Y_N:=\{w\in L^2([0,T], \tilde{H}^{1,1}): \|w\|_{L^2([0,T], \tilde{H}^{1,1})}\leqslant N\},$$
equipped with the weak topology on $L^2([0,T], \tilde{H}^{1,1})$. Then $Y_N$ is compact and metrizable, hence separable and complete.

Similarly, let
$$Z_N:=\{w\in L^\infty([0,T], \tilde{H}^{0,1}): \|w\|_{L^\infty([0,T], \tilde{H}^{0,1})}\leqslant N\},$$
equipped with the weak star topology on $L^\infty([0,T], \tilde{H}^{0,1})$. Then  $Z_N$ is compact and metrizable, hence separable and complete.

Define 
$$\chi_N=C([0,T],H^{-1})\bigcap L^2([0,T],H)\bigcap Y_N\bigcap Z_N:=X_1\cap X_2\cap X_3\cap X_4,$$
where $X_i$ are complete  separable metric spaces with metric $d_i$, $i=1,2,3,4.$ Let $\chi_N$ be equipped with the metric $d=\max\{d_1,d_2,d_3,d_4\}$. Then $\chi_N$ is separable. To show that $\chi_N$ is complete, it is enough to show that if $w_k\in \chi_N, k\in\mathbb{N}$ and $w_k\rightarrow w^{(i)}\in X_i$ in $d_i$ for every $1\leqslant i\leqslant 4$, then $w^{(1)}=w^{(2)}=w^{(3)}=w^{(4)}$. This is true since obviously we have the continuous embedding
$$X_i\subset \mathcal{M}([0,T], H^{-2}), \quad 1\leqslant i\leqslant 4,$$
where $\mathcal{M}$ denotes the space of  Radon measures. Hence $(\chi_N, d)$ is a complete separable metric space. Furthermore, the following embeddings are continuous and hence measurable:
$$(\chi_N,d)\subset X.$$
Therefore by Kuratowski's theorem we have for the Borel $\sigma$-algebra $\mathcal{B}(\chi_N)$ of $(\chi_N,d)$, 
$$\chi_N\in \mathcal{B}(X), \quad \mathcal{B}(\chi_N)=\mathcal{B}(X)\cap \chi_N.$$
Consequently, $\chi=\cup \chi_N\in \mathcal{B}(X)$.

Note that $\chi_N$ is a $\tau_\chi$-closed subset of $\chi$.
Let $A\subset \chi$ be $\tau_\chi$-closed. Then $A\cap \chi_N$ is $\tau_\chi$-closed too, hence
\begin{align*}
A\cap\chi_N&\in  \mathcal{B}(\chi_N)\\
&=\mathcal{B}(X)\cap\chi_N=\{B\in\mathcal{B}(X):B\subset \chi_N\}\\
&\subset \{B\in\mathcal{B}(X):B\subset \chi\}\\
&\subset \mathcal{B}(X)\cap \chi.
\end{align*}
Hence 
$$A=\bigcup^\infty_{N=1}A\cap \chi_N\in \mathcal{B}(X)\cap\chi$$
and
$$\mathcal{B}(\tau_\chi) \subset \mathcal{B}(X)\cap\chi.$$

Since $\chi\subset X$ continuously, hence measurably, we have $\mathcal{B}(X)\cap\chi \subset\mathcal{B}(\tau_\chi) .$
Then
 $$\mathcal{B}(\tau_\chi) =\mathcal{B}(X)\cap\chi.$$ Thus any probability measure on $X$ can be restricted on $\chi$.
 
Let $k$ be the same constant as in the proof of (\ref{eq02 in lemma estimate eq. for weak convergence MDP}) and let
\begin{align*}
K_R:=&\Big{\{} u\in C([0,T],H^{-1}): \sup_{t\in[0,T]}\|u(t)\|^2_H+\int^T_0\|u(t)\|^2_{\tilde{H}^{1,0}}dt+\|u\|_{C^\frac{1}{16}([0,T], H^{-1})}\\
&+\sup_{t\in[0,T]}e^{-k\int^t_0\|\partial_1u(s)\|^2_Hds}\|u(t)\|^2_{\tilde{H}^{0,1}}+\int^T_0e^{-k\int^t_0\|\partial_1u(s)\|^2_Hds}\|u(t)\|^2_{\tilde{H}^{1,1}}dt\leqslant R\Big{\}},
\end{align*}
where $C^\frac{1}{16}([0,T], H^{-1})$ is the H\"older space with the norm:
$$\|f\|_{C^\frac{1}{16}([0,T], H^{-1})}=\sup_{0\leqslant s< t\leqslant T}\frac{\|f(t)-f(s)\|_{H^{-1}}}{|t-s|^{\frac{1}{16}}}.$$

Then from the proof of \cite[Lemma 4.3]{LZZ18}, we know that for any $R>0$, $K_R$ is relatively compact in $\chi$. 

Now we only need to show that  for any $\delta>0$, there exists $R>0$, such that $P(X^\varepsilon\in  K_R)>1-\delta$ for any $\varepsilon\in(0,\varepsilon_0)$, where $\varepsilon_0$ is the constant such that Lemma \ref{estimate eq. for weak convergence} hold.

By Lemma \ref{estimate eq. for weak convergence} and Chebyshev inequality, we can choose $R_0$ large enough such that
\begin{align*}
P\left(\sup_{t\in[0,T]}\|X^\varepsilon(t)\|^2_H+\int^T_0\|X^\varepsilon(t)\|^2_{\tilde{H}^{1,0}}dt>\frac{R_0}{3}\right)<\frac{\delta}{4},
\end{align*}
and
\begin{align*}
P\left(\sup_{t\in[0,T]}e^{-k\int^t_0\|\partial_1X^\varepsilon(s)\|^2_Hds}\|X^\varepsilon(t)\|^2_{\tilde{H}^{0,1}}+\int^T_0e^{-k\int^t_0\|\partial_1X^\varepsilon(s)\|^2_Hds}\|X^\varepsilon(t)\|^2_{\tilde{H}^{1,1}}dt>\frac{R_0}{3}\right)<\frac{\delta}{4},
\end{align*}
where $k$ is the same constant as  in (\ref{eq02 in lemma estimate eq. for weak convergence MDP}).

 Fix $R_0$ and let
 \begin{align*}
 \hat{K}_{R_0}=&\Big{\{}u\in C([0,T],H^{-1}): \sup_{t\in[0,T]}\|u(t)\|^2_H+\int^T_0\|u(t)\|^2_{\tilde{H}^{1,0}}dt\leqslant \frac{R_0}{3}\text{ and }\\
&\sup_{t\in[0,T]}e^{-k\int^t_0\|\partial_1u(s)\|^2_Hds}\|u(t)\|^2_{\tilde{H}^{0,1}}+\int^T_0e^{-k\int^t_0\|\partial_1u(s)\|^2_Hds}\|u(t)\|^2_{\tilde{H}^{1,1}}dt\leqslant \frac{R_0}{3}\Big{\}}.
 \end{align*}
Then $P(X^\varepsilon\in C([0,T], H^{-1})\setminus \hat{K}_{R_0})<\frac{\delta}{2}$.

 Now for  $X^\varepsilon\in\hat{K}_{R_0}$, we have $\partial_1^2X^\varepsilon$ is uniformly bounded in $L^2([0,T],H^{-1})$. Similar as in Lemma \ref{wellposedness for skeleton eq. MDP}, $X^\varepsilon$ is uniformly bounded in $L^4([0,T],H^\frac{1}{2})$ and $L^4([0,T], L^4(\mathbb{T}^2))$, thus $B(X^\varepsilon, u^0+\sqrt{\varepsilon}\lambda(\varepsilon)X^\varepsilon)$ and $B(u^0, X^\varepsilon)$ are uniformly bounded in $L^2([0,T],H^{-1})$.  By H\"older's inequality, we have
 \begin{align*}
& \sup_{s,t\in[0,T],s\neq t}\frac{\|\int^t_s\partial_1^2X^\varepsilon(r)+B(X^\varepsilon, u^0+\sqrt{\varepsilon}\lambda(\varepsilon)X^\varepsilon)+B(u^0, X^\varepsilon)dr\|^2_{H^{-1}}}{|t-s|}\\
 \leqslant& \int^T_0\|\partial_1^2X^\varepsilon(r)+B(X^\varepsilon, u^0+\sqrt{\varepsilon}\lambda(\varepsilon)X^\varepsilon)+B(u^0, X^\varepsilon)\|^2_{H^{-1}}dr\leqslant C(R_0),
 \end{align*}
 where $C(R_0)$ is a constant depend on $R_0$.  For any $p\in(1,\frac{4}{3})$, by H\"older's inequality, we have
\begin{align*}
&\sup_{s,t\in[0,T],s\neq t}\frac{\|\int^t_s\sigma(r,u^0+\sqrt{\varepsilon}\lambda(\varepsilon)X^\varepsilon(r))v^\varepsilon(r) dr\|^p_{H^{-1}}}{|t-s|^{p-1}}\\
\leqslant& \int^T_0\|\sigma(r,u^0+\sqrt{\varepsilon}\lambda(\varepsilon)X^\varepsilon(r))v^\varepsilon(r)\|^p_{H^{-1}}dr\\
\leqslant& \int^T_0 \|\sigma(r,u^0+\sqrt{\varepsilon}\lambda(\varepsilon)X^\varepsilon(r))\|^p_{L_2(l^2,H^{-1})}\|v^\varepsilon(r)\|^p_{l^2}dr\\
\leqslant& C\int^T_0(1+\|u^0+\sqrt{\varepsilon}\lambda(\varepsilon)X^\varepsilon(r)\|^4_{H}+\|v^\varepsilon(r)\|^4_{l^2})dr\\
\leqslant& C(R_0),
\end{align*}
where we used Young's inequality and (A0) in the third inequality.

Moreover, for any $0\leqslant s\leqslant t\leqslant T$, by H\"older's inequality we have
\begin{align*}
&E\|\int^t_s\sigma(r,u^0+\sqrt{\varepsilon}\lambda(\varepsilon)X^\varepsilon(r))dW(r)\|^4_{H^{-1}}\\
\leqslant& C E\left(\int^t_s \|\sigma(r,u^0+\sqrt{\varepsilon}\lambda(\varepsilon)X^\varepsilon(r))\|^2_{L_2(l^2,H^{-1})}dr \right)^2\\
\leqslant& C|t-s|E\int^t_s \|\sigma(r,u^0+\sqrt{\varepsilon}\lambda(\varepsilon)X^\varepsilon(r))\|^4_{L_2(l^2,H^{-1})}dr\\
\leqslant& C|t-s|^2(1+E(\sup_{t\in[0,T]}\|u^0+\sqrt{\varepsilon}\lambda(\varepsilon)X^\varepsilon(t)\|^4_H))\\
\leqslant &C|t-s|^2,
\end{align*}
where we used (A0) in the third inequality and (\ref{eq01 in lemma estimate eq. for weak convergence MDP}) in the last inequality. Then by Kolmogorov's continuity criterion, for any $\alpha\in(0, \frac{1}{4})$, we have
\begin{align*}
E\left(\sup_{s,t\in[0,T],s\neq t}\frac{\|\int^t_s\sigma(r,u^0+\sqrt{\varepsilon}\lambda(\varepsilon)X^\varepsilon(r))dW(r)\|^4_{H^{-1}}}{|t-s|^{2\alpha}}\right)\leqslant C.
\end{align*}
Choose $p=\frac{8}{7}, \alpha=\frac{1}{8}$ in the above estimates, we deduce that there exists $R>R_0$ such that 
\begin{align*}
P\left(\|X^\varepsilon\|_{C^\frac{1}{16}([0,T],H^{-1})}>\frac{R}{3}, X^\varepsilon\in \hat{K}_{R_0}\right)\leqslant \frac{E\left(\sup_{s,t\in[0,T],s\neq t}\frac{\|X^\varepsilon(t)-X^\varepsilon(s)\|_{H^{-1}}}{|t-s|^{\frac{1}{16}}}1_{\{X^\varepsilon\in \hat{K}_{R_0}\}}\right)}{\frac{R}{3}}< \frac{\delta}{2}.
\end{align*}
Combining the fact that  $P(X^\varepsilon\in C([0,T], H^{-1})\setminus \hat{K}_{R_0})<\frac{\delta}{2}$, we finish the proof.

\end{proof}

\begin{lemma}\label{weak convergence MDP}
Let $\{v^\varepsilon\}_{\varepsilon>0}\subset \mathcal{A}_N$ for some $N<\infty$. Assume $v^\varepsilon$ converge to $v$ in distribution as $S_N$-valued random elements, then
\begin{align*}
g^\varepsilon\left(W(\cdot)+\lambda(\varepsilon)\int^\cdot_0v^\varepsilon(s)ds\right)\rightarrow g^0\left(\int^\cdot_0v(s)ds\right)
\end{align*}
in distribution as $\varepsilon\rightarrow 0$.
\end{lemma}

\begin{proof}
The proof follows essentially the same argument as in \cite[Proposition 4.7]{WZZ}. 

By Lemma \ref{Girsanov thm to prove existence MDP}, we have $X^\varepsilon=g^\varepsilon\left(W(\cdot)+\lambda(\varepsilon)\int^\cdot_0v^\varepsilon(s)ds\right)$. By a similar  argument as in the proof of Lemmas \ref{wellposedness for skeleton eq. MDP} and \ref{estimate eq. for weak convergence}, there exists a unique strong solution $$Y^\varepsilon\in L^{\infty}([0,T],\tilde{H}^{0,1})\bigcap L^2([0,T],\tilde{H}^{1,1})\bigcap C([0,T],H^{-1})$$ satisfying
\begin{align*}
dY^\varepsilon(t)=&\partial_1^2Y^\varepsilon(t)dt+\lambda^{-1}(\varepsilon)\sigma(t,u^0+\sqrt{\varepsilon}\lambda(\varepsilon) X^\varepsilon(t))dW(t),\\
Y^\varepsilon(0)=&0,
\end{align*}
and 
\begin{align*}
\lim_{\varepsilon\rightarrow 0}\left[E\sup_{t\in[0,T]}\|Y^\varepsilon(t)\|_{H}^2+E\int^T_0\|Y^\varepsilon(t)\|^2_{\tilde{H}^{1,0}}dt\right]=0,
\end{align*}
\begin{align*}
\lim_{\varepsilon\rightarrow 0}\left[E\sup_{t\in[0,T]}(e^{-kg(t)}\|Y^\varepsilon(t)\|_{\tilde{H}^{0,1}}^2)+E\int^T_0e^{-kg(t)}\|Y^\varepsilon(t)\|^2_{\tilde{H}^{1,1}}dt\right]=0,
\end{align*}
where $g(t)=\int^t_0\|\partial_1 X^\varepsilon(s)\|^2_Hds$ and $k$ are the same as in (\ref{eq02 in lemma estimate eq. for weak convergence MDP}).

Set $$\Xi:=\left(\chi, \mathcal{S}_N,  L^\infty([0,T],H)\bigcap L^2([0,T],\tilde{H}^{1,0}) \bigcap C([0,T].H^{-1})\right).$$ The above limit implies that $Y^\varepsilon\rightarrow 0$  in $ L^\infty([0,T],H)\bigcap L^2([0,T],\tilde{H}^{1,0}) \bigcap C([0,T].H^{-1})$ almost surely as $\varepsilon\rightarrow 0$ (in the sense of subsequence). By Lemma \ref{tightness lemma} the family $\{(X^\varepsilon,v^\varepsilon)\}_{\varepsilon\in(0,\varepsilon_0)}$ is tight in $(\chi, \mathcal{S}_N)$. Let $(X_v, v, 0)$ be any limit point of $\{(X^\varepsilon, v^\varepsilon, Y^\varepsilon)\}_{\varepsilon\in(0,\varepsilon_0)}$. Our goal is to show that $X_v$ has the same law as $g^0\left(\int^\cdot_0v(s)ds\right)$ and $X^\varepsilon$ convergence in distribution to $X_v$ in the space $ L^{\infty}([0,T],H)\bigcap L^2([0,T],\tilde{H}^{1,0})\bigcap C([0,T],H^{-1})$.

By  Jakubowski-Skorokhod's representation theorem (see \cite{Ja98} or \cite[Theorem 4.3]{LZZ18}), there exists a stochastic basis $(\tilde{\Omega}, \tilde{\mathcal{F}}, \{\tilde{\mathcal{F}}_t\}_{t\in[0,T]}, \tilde{P})$ and, on this basis, $\Xi$-valued random variables $(\tilde{X}_v, \tilde{v}, 0)$, $(\tilde{X}^\varepsilon, \tilde{v}^\varepsilon, \tilde{Y}^\varepsilon)$, such that $(\tilde{X}^\varepsilon, \tilde{v}^\varepsilon, \tilde{Y}^\varepsilon)$ (respectively $(\tilde{X}_v, \tilde{v}, 0)$) has the same law as $(X^\varepsilon,v^\varepsilon, Y^\varepsilon)$ (respectively $(X_v,v,0)$), and $(\tilde{X}^\varepsilon, \tilde{v}^\varepsilon, \tilde{Y}^\varepsilon)\rightarrow (\tilde{X}_v, \tilde{v}, 0)$, $\tilde{P}$-a.s.

We have
\begin{equation}\label{weak convergence step 1 MDP}\aligned
d(\tilde{X}^\varepsilon(t)-\tilde{Y}^\varepsilon(t))=&\partial_1^2(\tilde{X}^\varepsilon(t)-\tilde{Y}^\varepsilon(t))dt-B(\tilde{X}^\varepsilon, u^0+\sqrt{\varepsilon}\lambda(\varepsilon)\tilde{X}^\varepsilon)dt\\
&-B(u^0, \tilde{X}^\varepsilon)dt+\sigma(t,u^0+\sqrt{\varepsilon}\lambda(\varepsilon)\tilde{X}^\varepsilon(t))\tilde{v}^\varepsilon(t)dt,\\
\tilde{X}^\varepsilon(0)-\tilde{Y}^\varepsilon(0)=&0,
\endaligned
\end{equation}
and
\begin{align*}
&P(\tilde{X}^\varepsilon-\tilde{Y}^\varepsilon\in L^\infty([0,T],H)\bigcap L^2([0,T], \tilde{H}^{1,0})\bigcap C([0,T],H^{-1}))\\
=&P({X}^\varepsilon-{Y}^\varepsilon\in L^\infty([0,T],H)\bigcap L^2([0,T], \tilde{H}^{1,0})\bigcap C([0,T],H^{-1}))\\
=&1.
\end{align*}

Let $\tilde{\Omega}_0$ be the subset of $\tilde{\Omega}$ such that for $\omega\in\tilde{\Omega}_0$,
$$(\tilde{X}^\varepsilon, \tilde{v}^\varepsilon, \tilde{Y}^\varepsilon)(\omega)\rightarrow (\tilde{X}_v,\tilde{v}, 0)(\omega) \text{ in }\Xi,$$ 
and 
$$e^{-k\int^{\cdot}_0\|\tilde{X}^\varepsilon(\omega, s)\|^2_Hds}\tilde{Y}^\varepsilon(\omega)\rightarrow 0 \text{ in } L^{\infty}([0,T],\tilde{H}^{0,1})\bigcap L^2([0,T],\tilde{H}^{1,1})\bigcap C([0,T],H^{-1}),$$ 
then $P(\tilde{\Omega}_0)=1$. For any $\omega\in\tilde{\Omega}_0$, fix $\omega$, we have $\sup_{\varepsilon}\int^T_0\|\tilde{X}^\varepsilon(\omega,s)\|_H^2ds<\infty$, then we deduce that 
\begin{equation}\label{convergence of Yvarepsilon to 0 MDP}
\lim_{\varepsilon\rightarrow 0} \left(\sup_{t\in[0,T]}\|\tilde{Y}^\varepsilon(\omega,t)\|_{\tilde{H}^{0,1}}+\int^T_0\|\tilde{Y}^\varepsilon(\omega,t)\|^2_{\tilde{H}^{1,1}}dt\right)=0.
\end{equation}

Now we  show that
\begin{equation}\label{weak convergence step 2 MDP}
\sup_{t\in[0,T]}\|\tilde{X}^\varepsilon(\omega,t)-\tilde{X}_v(\omega,t)\|^2_H+\int^T_0\|\tilde{X}^\varepsilon(\omega, t)-\tilde{X}_v(\omega,t)\|^2_{\tilde{H}^{1,0}}dt\rightarrow 0\text{ as }\varepsilon\rightarrow 0.
\end{equation}

Let  $U^\varepsilon=\tilde{X}^\varepsilon(\omega)-\tilde{Y}^\varepsilon(\omega)$, then by  (\ref{weak convergence step 1 MDP}) we have
\begin{equation}\aligned
dU^\varepsilon(t)=&\partial_1^2U^\varepsilon(t) dt-B(U^\varepsilon+\tilde{Y}^\varepsilon, u^0+\sqrt{\varepsilon}\lambda(\varepsilon)(U^\varepsilon+\tilde{Y}^\varepsilon))dt\\
&-B(u^0, U^\varepsilon+\tilde{Y}^\varepsilon)+\sigma(t,u^0+\sqrt{\varepsilon}\lambda(\varepsilon) (U^\varepsilon(t)+\tilde{Y}^\varepsilon(t)))\tilde{v}^\varepsilon(t)dt.
\endaligned
\end{equation}
Since $U^\varepsilon(\omega)\rightarrow \tilde{X}_v(\omega)$  in $\chi$, by a very similar argument as in Lemma \ref{good rate function MDP} we deduce that  $\tilde{X}_v=X^{\tilde{v}}=g^0\left(\int^\cdot_0{\tilde{v}}(s)ds\right)$. Moreover, note that $\tilde{X}^\varepsilon(\omega)\rightarrow X^{\tilde{v}}(\omega)$ weak star in $L^{\infty}([0,T],\tilde{H}^{0,1})$, then the uniform boundedness principle implies that 
\begin{equation}\label{uniform bded of Z in H01 MDP}
\sup_{\varepsilon}\sup_{t\in[0,T]}\|\tilde{X}^\varepsilon(\omega)\|_{\tilde{H}^{0,1}}<\infty.
\end{equation}
Let $w^\varepsilon=U^\varepsilon-X^{\tilde{v}}$, then we have
\begin{align*}
&\|w^\varepsilon(t)\|^2_H+2\int^t_0\|\partial_1w^\varepsilon(s)\|^2_Hds\\
=&-2\int^t_0\langle w^\varepsilon(s), B(U^\varepsilon+\tilde{Y}^\varepsilon, u^0+\sqrt{\varepsilon}\lambda(\varepsilon)(U^\varepsilon+\tilde{Y}^\varepsilon))-B(X^{\tilde{v}},u^0)\rangle ds\\
&-2\int^t_0\langle w^\varepsilon(s), B(u^0, w^\varepsilon+\tilde{Y}^\varepsilon)\rangle ds\\
&+2\int^t_0\langle w^\varepsilon(s), \sigma(s,u^0+\sqrt{\varepsilon}\lambda(\varepsilon)(U^\varepsilon+\tilde{Y}^\varepsilon)){\tilde{v}}^\varepsilon(s)-\sigma(s,u^0){\tilde{v}}(s)\rangle ds\\
=:&I_1+I_2+I_3.
\end{align*}
By Lemma \ref{b(u,v,w)}, we have
\begin{align*}
&|I_1+I_2|\\
=&|\int^t_0b(w^\varepsilon, u^0+\sqrt{\varepsilon}\lambda(\varepsilon)(X^{\tilde{v}}+\tilde{Y}^\varepsilon), w^\varepsilon)+b(\tilde{Y}^\varepsilon, u^0, w^\varepsilon)\\
&+\sqrt{\varepsilon}\lambda(\varepsilon) b(X^{\tilde{v}}+\tilde{Y}^\varepsilon, X^{\tilde{v}}+\tilde{Y}^\varepsilon, w^\varepsilon)+b(u^0, \tilde{Y}^\varepsilon, w^\varepsilon)ds|\\
\leqslant &\int^t_0[\frac{1}{2}\|\partial_1w^\varepsilon(s)\|^2_H+C(1+\|u^0(s)\|^2_{\tilde{H}^{1,1}}+\|X^{\tilde{v}}(s)\|_{\tilde{H}^{1,1}}^2+\|\tilde{Y}^\varepsilon(s)\|^2_{\tilde{H}^{1,1}})\|w^\varepsilon(s)\|^2_H]ds\\
&+\int^t_0[\|\tilde{Y}^\varepsilon(s)\|^2_{\tilde{H}^{1,0}}+C\|u^0(s)\|^2_{\tilde{H}^{1,1}}\|w^\varepsilon(s)\|^2_H]ds\\
&+\sqrt{\varepsilon}\lambda(\varepsilon)\int^t_0[\|X^{\tilde{v}}(s)\|^2_{\tilde{H}^{1,0}}+\|\tilde{Y}^\varepsilon(s)\|^2_{\tilde{H}^{1,0}}+(\|X^{\tilde{v}}(s)\|_{\tilde{H}^{1,1}}^2+\|\tilde{Y}^\varepsilon(s)\|^2_{\tilde{H}^{1,1}})\|w^\varepsilon(s)\|^2_H]ds\\
&+\int^t_0[\|\tilde{Y}^\varepsilon(s)\|^2_{\tilde{H}^{1,1}}+C\|u^{0}(s)\|^2_{\tilde{H}^{1,0}}\|w^\varepsilon(s)\|^2_H]ds\\
\leqslant &\int^t_0[\frac{1}{2}\|\partial_1w^\varepsilon(s)\|^2_H+C(1+\|u^0(s)\|^2_{\tilde{H}^{1,1}}+\|X^{\tilde{v}}(s)\|_{\tilde{H}^{1,1}}^2)\|w^\varepsilon(s)\|^2_H]ds\\
+& C\int^t_0\|\tilde{Y}^\varepsilon(s)\|^2_{\tilde{H}^{1,1}}ds+\sqrt{\varepsilon}\lambda(\varepsilon)\int^t_0\|X^{\tilde{v}}(s)\|^2_{\tilde{H}^{1,0}}ds.
\end{align*}
where we  used the fact that by \eqref{convergence of Yvarepsilon to 0 MDP} and \eqref{uniform bded of Z in H01 MDP} $w^\varepsilon$ are uniformly bounded in $L^\infty([0,T], H)$ in the last inequality.
By (A1) and (A3)  we have
\begin{align*}
|I_3(t)|=&\int^t_0\langle w^\varepsilon(s), (\sigma(s,u^0+\sqrt{\varepsilon}\lambda(\varepsilon)[U^\varepsilon+\tilde{Y}^\varepsilon])-\sigma(s,u^0)) {\tilde{v}}^\varepsilon(s)\rangle ds\\
&+\int^t_0\langle w^\varepsilon(s), \sigma(s,u^0)({\tilde{v}}^\varepsilon(s)- {\tilde{v}}(s))\rangle ds\\
\leqslant &C(\sqrt{\varepsilon}\lambda(\varepsilon))^\frac{1}{2}\int^t_0(\|w^\varepsilon(s)\|_H\|{\tilde{v}}^\varepsilon(s)\|_{l^2}(\|w^
\varepsilon(s)\|^2_{\tilde{H}^{1,0}}+\|X^{\tilde{v}}(s)\|^2_{\tilde{H}^{1,0}}+\|\tilde{Y}^\varepsilon(s)\|^2_{\tilde{H}^{1,0}})^\frac{1}{2} ds\\
&+\int^t_0\|w^\varepsilon(s)\|_H\|{\tilde{v}}^\varepsilon(s)-{\tilde{v}}(s)\|_{l^2}(K_0+K_1\|u^0(s)\|^2_H+K_2\|\partial_1 u^0(s)\|^2_H)^\frac{1}{2} ds\\
\leqslant & (\sqrt{\varepsilon}\lambda(\varepsilon))^\frac{1}{2}\left(CN+C_1\int^t_0(\|w^
\varepsilon(s)\|^2_{\tilde{H}^{1,0}}+\|X^{\tilde{v}}(s)\|^2_{\tilde{H}^{1,0}}+\|\tilde{Y}^\varepsilon(s)\|^2_{\tilde{H}^{1,0}} ds\right)\\
&+CN^\frac{1}{2}\left(\int^t_0\|w^\varepsilon(s)\|_H^2(K_0+K_1\|u^0(s)\|^2_H+K_2\|\partial_1 u^0(s)\|^2_H) ds\right)^\frac{1}{2},
\end{align*}
where we used the fact that $w^\varepsilon$ are uniformly bounded in $L^\infty([0,T], H)$ and that ${\tilde{v}}^\varepsilon$,  ${\tilde{v}}$ are in $\mathcal{A}_N$. Note here $C_1$ is a positive constant.
Thus choose $\varepsilon$ small enough such that $\frac{1}{2}+(\sqrt{\varepsilon}\lambda(\varepsilon))^\frac{1}{2}C_1<1$,  we have
\begin{align*}
&\|w^\varepsilon(t)\|^2_H+\int^t_0\|\partial_1w^\varepsilon(s)\|^2_Hds\\
\leqslant& C\int^t_0(1+\|u^0(s)\|^2_{\tilde{H}^{1,1}}+\|X^{\tilde{v}}(s)\|_{\tilde{H}^{1,1}}^2)\|w^\varepsilon(s)\|^2_Hds\\
&+C\int^t_0\|\tilde{Y}^\varepsilon(s)\|^2_{\tilde{H}^{1,1}}ds+\sqrt{\varepsilon}\lambda(\varepsilon)\int^t_0\|X^{\tilde{v}}(s)\|^2_{\tilde{H}^{1,0}}ds\\
&+C (\sqrt{\varepsilon}\lambda(\varepsilon))^\frac{1}{2}\left(N+\int^t_0(\|w^
\varepsilon(s)\|^2_H+\|X^{\tilde{v}}(s)\|^2_{\tilde{H}^{1,0}}+\|\tilde{Y}^\varepsilon(s)\|^2_{\tilde{H}^{1,0}} ds\right)\\
&+CN^\frac{1}{2}\left(\int^t_0(1+\|u^0(s)\|^2_{\tilde{H}^{1,1}})\|w^\varepsilon(s)\|^2_Hds\right)^\frac{1}{2}.
\end{align*}

Since $U^\varepsilon(\omega)\rightarrow X^{\tilde{v}}(\omega)$ strongly in $L^2([0,T],H)$ and $\tilde{Y}^\varepsilon\rightarrow 0$ in $L^2([0,T],\tilde{H}^{1,1})$, the same argument used in Lemma \ref{good rate function MDP} implies
\begin{equation}\label{weak convergence step 3 MDP}
\sup_{t\in[0,T]}\|\tilde{X}^\varepsilon(\omega,t)-X^{\tilde{v}}(\omega,t)\|^2_H+\int^T_0\|\tilde{X}^\varepsilon(\omega, t)-X^{\tilde{v}}(\omega,t)\|^2_{\tilde{H}^{1,0}}dt\rightarrow 0\text{ as }\varepsilon\rightarrow 0.
\end{equation}

The proof is thus complete.

\end{proof}

\begin{proof}[{Proof of Theorem \ref{main result MDP}}]
The result holds from Lemmas \ref{weak convergence method}, \ref{good rate function MDP} and \ref{weak convergence MDP}.
\end{proof}

\appendix
\section{Appendix}

We first  present several  lemmas from \cite[Appendix]{CZ20}:

\begin{lemma}\label{b(u,v,w)}
For smooth functions $u,v,w$ form $\mathbb{T}^2$ to $\mathbb{R}^2$ with divergence free condition, we have
$$|b(u,v,w)|\leqslant C\|u\|_{H^{1,0}}\|v\|_{H^{1,1}}\|w\|_{L^2}.$$
\end{lemma}

\begin{proof}
\begin{align*}
&|b(u, v, w)|\\
\leqslant &(\|u^1\|_{L^\infty_h(L^2_v)}\|\partial_1v\|_{L^2_h(L^\infty_v)}+\|u^2\|_{L^2_h(L^\infty_v)}\|\partial_2 v\|_{L^\infty_h(L^2_v)})\|w\|_{L^2}\\
\leqslant&C\Big{(}(\|u^1\|_{L^2}\|\partial_1u^1\|_{L^2}+\|u^1\|^2_{L^2})^\frac{1}{2}(\|\partial_1v\|_{L^2}\|\partial_1\partial_2v\|_{L^2}+\|\partial_1v\|^2_{L^2})^\frac{1}{2}\\
&+(\|u^2\|_{L^2}\|\partial_2u^2\|_{L^2}+\|u^2\|_{L^2}^2)^\frac{1}{2}(\|\partial_2v\|_{L^2}\|\partial_1\partial_2v\|_{L^2}+\|\partial_2v\|_{L^2}^2)^\frac{1}{2}\Big{)}\|w\|_{L^2}\\
\leqslant&C\|u\|_{H^{1,0}}\|v\|_{H^{1,1}}\|w\|_{L^2},
\end{align*}
where we used the divergence free condition to deal with the term $\partial_2u^2$ in the last inequality.
\end{proof}

\begin{lemma}\label{estimate for b with partial_2}
For smooth function $u$ form $\mathbb{T}^2$ to $\mathbb{R}^2$ with divergence free condition, we have
$$|\langle \partial_2u, \partial_2(u\cdot \nabla u)\rangle|\leqslant a\|\partial_1\partial_2 u\|^2_{L^2}+C(1+\|\partial_1 u\|^2_{L^2})\|\partial_2 u\|^2_{L^2},$$
where $a>0$ is a constant small enough.
\end{lemma}

 The following estimates are obtained by \cite{CDGG00} in dimension 3, we now present its 2-dimension version.

\begin{lemma}[{\cite[Lemma 3]{CDGG00}}]\label{commutator estimates}
For any real number $s_0>\frac{1}{2}$ and $s\geqslant s_0$,  for any vector fields $u$ and $ w$, with divergence free condition, there exists  constants $C$ and $d_k(u,w)$ such that
\begin{align*}
|\langle \Delta^v_k(u\cdot \nabla w), \Delta^v_k w\rangle|\leqslant Cd_k2^{-2ks}\|w\|_{H^{\frac{1}{4},s}}(&\|u\|_{H^{\frac{1}{4},s_0}}\|\partial_1w\|_{H^{0,s}}+\|u\|_{H^{\frac{1}{4},s}}\|\partial_1w\|_{H^{0,s_0}}\\
&+\|\partial_1u\|_{H^{0,s_0}}\|w\|_{H^{\frac{1}{4},s}}+\|\partial_1u\|_{H^{0,s}}\|w\|_{H^{\frac{1}{4},s_0}}),
\end{align*}
where $\sum_{k}d_k= 1$.
\end{lemma}

\begin{proof}
Define 
$$F^h_k=\Delta^v_k(u^1\partial_1w)\text{ and }F^v_k=\Delta^v_k(u^2\partial_2w).$$

Let us start  by proving the result for $F^h_k$.  Recall the Bony decomposition (see \cite{BCD11}) in vertical variables for tempered distributions $a,b$:
$$ab=T^v_ab+T^v_ba+R^v(a,b),$$
with
$$T^v_ab=\sum_{j}S^v_{j-1}a\Delta^v_jb\quad \text{ and }\quad R^v(a,b)=\sum_{|k-j|\leqslant 1}\Delta^v_ka\Delta^v_jb,$$ 
where $S^v_{j-1}a=\sum_{j'\leqslant j-2}\Delta^v_{j'}a$.

Then we have by H\"older's inequality and Sobolev embedding $H^{\frac{1}{4}}(\mathbb{T})\hookrightarrow L^{4}(\mathbb{T})$
\begin{equation}\label{F h k}\aligned
\langle \Delta^v_k(u^1\partial_1w), \Delta^v_k w \rangle\leqslant& \|\Delta^v_k(u^1\partial_1w)\|_{L^2_v(L_h^{\frac{4}{3}})}\|\Delta^v_k w\|_{L^2_v(L_h^4)}\\
\leqslant &C\|\Delta^v_k(T^v_{u^1}\partial_1w+T^v_{\partial_1w}u^1+R^v(u^1,\partial_1w))\|_{L^2_v(L_h^{\frac{4}{3}})}\|\Delta^v_kw\|_{L^2_v(H_h^{\frac{1}{4}})}\\
\leqslant & C\|\Delta^v_k(T^k_{u^1}\partial_1w+T^v_{\partial_1w}u^1+R^v(u^1,\partial_1w))\|_{L^2_v(L_h^{\frac{4}{3}})} 2^{-ks} c_k\|w\|_{H^{\frac{1}{4}, s}},
\endaligned
\end{equation}
where $c_k=\frac{2^{ks}\|\Delta^k_vw \|_{L^2_v(H_h^{\frac{1}{4}})}}{\|w\|_{H^{\frac{1}{4}, s}}}\in l^2$.
 For the first  term of the third line, we have 
\begin{align*}
&\|\Delta^v_k(T^k_{u^1}\partial_1w)\|_{L^2_v(L_h^\frac{4}{3})}\\
\leqslant& \sum_{|k-k'|\leqslant N_0} \|S^v_{k'-1} u^1\Delta^v_{k'}\partial_1w\|_{L^2_v(L_h^{\frac{4}{3}})}\leqslant\sum_{|k-k'|\leqslant N_0} \|S^v_{k'-1} u^1\|_{L^\infty_v(L_h^4)}\|\Delta^v_{k'}\partial_1w\|_{L^2_v(L_h^{2})} \\
\leqslant &C \sum_{|k-k'|\leqslant N_0}  \|u^1\|_{H^{\frac{1}{4},s_0}} 2^{-k's}b_{k'}\|\partial_1w\|_{H^{0,s}} \leqslant C b^{(1)}_k2^{-ks}\|u^1\|_{H^{\frac{1}{4},s_0}}\|\partial_1w\|_{H^{0,s}},
\end{align*}
where $b_k=\frac{2^{ks}\|\Delta^v_{k}\partial_1w\|_{L^2_v(L_h^{2})} }{\|\partial_1w\|_{H^{0,s}}}\in l^2$ and $b^{(1)}_k=2^{ks}\sum_{|k-k'|\leqslant N_0} 2^{-k's}b_{k'}\in l^2$. Note here $N_0$ depends on the choice of Dyadic partition. For the second term, similarly we have
\begin{align*}
\|\Delta^v_k(T^k_{\partial_1w}u^1)\|_{L^2_v(L_h^\frac{4}{3})}\leqslant& \sum_{|k-k'|\leqslant N_0} \|S^v_{k'-1}\partial_1w \|_{L^\infty_v(L_h^2)}\|\Delta^v_{k'}u^1\|_{L^2_v(L_h^{4})} \\
\leqslant &C \sum_{|k-k'|\leqslant N_0}  \|\partial_1w\|_{H^{0,s_0}} 2^{-k's}a_{k'}\|u\|_{H^{\frac{1}{4},s}} \leqslant C a^{(1)}_k2^{-ks}\|\partial_1w\|_{H^{0,s_0}}\|u\|_{H^{\frac{1}{4},s}},
\end{align*}
where $a_k=\frac{2^{ks}\|\Delta^v_{k}u\|_{L^2_v(H_h^{\frac{1}{4}})} }{\|u\|_{H^{\frac{1}{4},s}}}\in l^2$ and $a^{(1)}_k=2^{ks}\sum_{|k-k'|\leqslant N_0} 2^{-k's}\tilde{c}_k\in l^2$. 
\begin{align*}
\|\Delta^v_kR^v(u^1,\partial_1w)\|_{L^2_v(L_h^{\frac{4}{3}})}
\leqslant &\sum_{|k'-j|\leqslant 1, k'\geqslant k- N_0}\|\Delta^v_{k'}u^1\|_{L^2_v(L_h^{{4}})}\|\Delta^v_j\partial_1w\|_{L^\infty_v(L^2_h)}\\
\leqslant& C\sum_{k'\geqslant k- N_0}2^{-k's} a_{k'}\|u\|_{H^{\frac{1}{4},s}}\|\partial_1w\|_{H^{0,s_0}}\\
 \leqslant & Ca^{(2)}_k2^{-ks}\|u\|_{H^{\frac{1}{4},s}}\|\partial_1 w\|_{H^{0,s_0}},
\end{align*}
where $a^{(2)}_k=2^{ks}\sum_{k'\geqslant k-N_0} 2^{-k's}a_{k'}=\sum_{k'\in \mathbb{Z}} I_{\{k'\leqslant N_0\}}2^{k's} a_{k-k'}$ and by Young's convolution inequality 
$$\|a^{(2)}\|_{l^2}\leqslant \|I_{\{k'\leqslant N_0\}}2^{k's}\|_{l^1}\|a\|_{l^2}<\infty.$$

This implies that 
$$|\langle F^h_k, \Delta^v_k w\rangle|\leqslant C c_k(b^{(1)}_k+a^{(1)}_k+a^{(2)}_k) 2^{-2ks}\|w\|_{H^{\frac{1}{4},s}}(\|u\|_{H^{\frac{1}{4},s_0}}\|\partial_1w\|_{H^{0,s}}+\|u\|_{H^{\frac{1}{4},s}}\|\partial_1 w\|_{H^{0,s_0}}),$$
where $  c_k(b^{(1)}_k+a^{(1)}_k+a^{(2)}_k)\in l^1$.

To estimate the term $\langle F^v_k, \Delta^v_kw\rangle$, write $\Delta^v_k(u^2\partial_2w) =F^{v,1}_k+ F^{v,2}_k$ with 
\begin{align*}
F^{v,1}_k= \Delta^v_k\sum_{k'\geqslant k-N_0} S^v_{k'+2}\partial_2w\Delta^v_{k'}u^2\quad\text{and}\quad 
F^{v,2}_k=\Delta^v_k \sum_{|k-k'|\leqslant N_0}S^v_{k'-1}u^2\Delta^v_{k'}\partial_2w.
\end{align*}

For $F^{v,1}_k$, again we have by   H\"older's inequality and Sobolev embedding, 
\begin{align*}
\|F^{v,1}_k\|_{L^2_v(L_h^{\frac{4}{3}})}\leqslant& \sum_{k'\geqslant k-N_0}\| S^v_{k'+2}\partial_2w\|_{L^\infty_v(L_h^4)}\|\Delta^v_{k'}u^2\|_{L^2_v(L_h^{2})}\\
\leqslant& C \sum_{k'\geqslant k-N_0}2^{k'}\| S^v_{k'+2}w\|_{L^\infty_v(L_h^{4})}2^{-k'}\|\Delta^v_{k'}\partial_2 u^2\|_{L^2_v(L^2_h)}\\
\leqslant &C \sum_{k'\geqslant k-N_0} \|w\|_{H^{\frac{1}{4},s_0}}2^{-k's}\tilde{c}_{k'}\|\partial_1u\|_{H^{0,s}}\\
\leqslant &C 2^{-ks} \tilde{c}^{(2)}_k\|w\|_{H^{\frac{1}{4},s_0}}\|\partial_1u\|_{H^{0,s}},
\end{align*}
where we use Bernstein's inequality twice in the second inequality and divergence free condition in the third inequality. Note here $\tilde{c}_{k}=\frac{2^{ks}\|\Delta^v_{k}\partial_1 u\|_{L^2_v(L^2_h)} }{\|\partial_1 u\|_{H^{0,s}}}\in l^2$ and $\tilde{c}^{(2)}_k=2^{ks}\sum_{k'\geqslant k-N_0} 2^{-k's}\tilde{c}_{k'}\in l^2$.

Then similar as \eqref{F h k} we have
$$|\langle F^{v,1}_k, \Delta^v_k w\rangle|\leqslant C c_k\tilde{c}^{(2)}_k 2^{-2ks}\|w\|_{H^{\frac{1}{4},s}} \|w\|_{H^{\frac{1}{4},s_0}}\|\partial_1u\|_{H^{0,s}}.$$
The last term $F^{v,2}_k$ requires commutator estimates.  Following a computation  in \cite{CL92}, we have
\begin{align*}
\langle F^{v,2}_k, \Delta^v_k w\rangle =& \langle S^v_{k-1}u^2\Delta^v_k\partial_2 w, \Delta^v_kw\rangle+R_k(u,w)\quad \text{with}\\
R_k(u,v)=&\sum_{|k-k'|\leqslant N_0}\langle [\Delta^v_k, S^v_{k'-1}u^2]\Delta^v_{k'}\partial_2w, \Delta^v_k w\rangle \\
&-\sum_{|k'-k|\leqslant N_0}\langle (S^v_{k-1}-S^v_{k'-1})u^2\Delta^v_k\Delta^v_{k'}\partial_2w, \Delta^v_kw\rangle.
\end{align*}

Using  an integration by parts and divergence free condition,  we have
\begin{equation}\label{commutator integration by part}\aligned
|\langle S^v_{k-1}u^2\Delta^v_k\partial_2 w, \Delta^v_kw\rangle|=&\frac{1}{2}|\langle S^v_k\partial_2u^2\Delta^v_kw,\Delta^v_kw\rangle|=\frac{1}{2}|\langle S^v_k\partial_1u^1\Delta^v_kw,\Delta^v_kw\rangle|\\
\leqslant & C \|S^v_k\partial_1u^1\|_{L^\infty_v(L^2_h)}\|\Delta^v_kw\|_{L^2_v(L^4_h)}^2\\
\leqslant &Cc^2_k2^{-2ks}\|\partial_1u\|_{H^{0,s_0}}\|w\|_{H^{\frac{1}{4},s}}^2.
\endaligned
\end{equation}

Note that the Fourier transform of  $(S^v_{k-1}-S^v_{k'-1})u^2$ is supported in $2^k\mathcal{A}$ since $|k-k'|\leqslant N_0$ where $\mathcal{A}$ is an annulus.  We have by Bernstein's inequality
\begin{align*}
&\|\sum_{|k'-k|\leqslant N_0} (S^v_{k-1}-S^v_{k'-1})u^2\Delta^v_k\Delta^v_{k'}\partial_2w\|_{L^2_v(L_h^{\frac{4}{3}})}\\
\leqslant &\sum_{|k'-k|\leqslant N_0} \| (S^v_{k-1}-S^v_{k'-1})u^2\|_{L^\infty_v(L^2_h)}\|\Delta^v_k\Delta^v_{k'}\partial_2w\|_{L^2_v(L^4_h)}\\
\leqslant &C\sum_{|k'-k|\leqslant N_0} 2^k\| (S^v_{k-1}-S^v_{k'-1})\partial_2 u^2\|_{L^\infty_v(L^2_h)}2^{-k}\|\Delta^v_kw\|_{L^2_v(L^4_h)}\\
\leqslant &C\sum_{|k'-k|\leqslant N_0} \|\partial_1 u^1\|_{H^{0,s_0}}2^{-ks} c_k\|w\|_{H^{\frac{1}{4},s}}.
\end{align*}
This similar as \eqref{F h k} implies that 
$$|\langle \sum_{|k'-k|\leqslant N_0} (S^v_{k-1}-S^v_{k'-1})u^2\Delta^v_{k'}\partial_2w, \Delta^v_kw\rangle|\leqslant  Cc^2_k2^{-2ks}\|\partial_1u\|_{H^{0,s_0}}\|w\|_{H^{\frac{1}{4},s}}^2.$$

To estimate the term $\langle [\Delta^v_k, S^v_{k'-1}u^2]\Delta^v_{k'}\partial_2w, \Delta^v_k w\rangle$, we have for any function $f$,
\begin{align*}
&[\Delta^v_k, S^v_{k'-1}u^2]f(x_1,x_2)\\
=& 2^k\int_{\mathbb{T}_v} h(2^ky_2)(S^v_{k'-1}u^2(x_1,x_2)-S^v_{k'-1}u^2(x_1,x_2-y_2))f(x_1,x_2-y_2)dy_2\\
=& \int_{\mathbb{T}_v\times [0,1]}h_1(2^ky_2)(S^v_{k'-1}\partial_2u^2)(x_1,x_2+(t-1)y_2)f(x_1,x_2-y_2)dy_2dt\\
=&-\int_{\mathbb{T}_v\times [0,1]}h_1(2^ky_2)(S^v_{k'-1}\partial_1u^1)(x_1,x_2+(t-1)y_2)f(x_1,x_2-y_2)dy_2dt,
\end{align*}
where $h=\mathcal{F}^{-1}\chi^{(1)}, (k=-1)$ or $h=\mathcal{F}^{-1}\theta^{(1)}, (k\geqslant 0)$, $h_1(z)=zh(z)$ and we use divergence free condition in the last line. This implies
$$\|[\Delta^v_k, S^v_{k'-1}u^2]f(\cdot, x_2)\|_{L_h^{\frac{4}{3}}}\leqslant C\int |h_1(2^ky_2)|\|S^v_{k'-1}\partial_1u^1\|_{L^\infty_v(L^2_h)}\|f(\cdot,x_2-y_2)\|_{L_h^{{4}}}dy_2$$

Then we get
\begin{align*}
\|[\Delta^v_k, S^v_{k'-1}u^2]f\|_{L^2_v(L_h^{\frac{4}{3}})}\leqslant C 2^{-k}\|S^v_{k'-1}\partial_1u^1\|_{L^\infty_v(L^2_h)}\|f\|_{L^2_v(H_h^{\frac{1}{4}})}.
\end{align*}

Hence 
\begin{align*}
&|\sum_{|k-k'|\leqslant N_0}\langle [\Delta^v_k, S^v_{k'-1}u^2]\Delta^v_{k'}\partial_2w, \Delta^v_k w\rangle |\\
\leqslant &C 2^{-k}\sum_{|k-k'|\leqslant N_0} \|S^v_{k'-1}\partial_1u^1\|_{L^\infty_v(L^2_h)} 2^{k'}\|\Delta^v_{k'}w\|_{L^2_v(H_h^{\frac{1}{4}})}\|\Delta^v_k w\|_{L^2_v(H_h^{\frac{1}{4}})}\\
\leqslant &C \sum_{|k-k'|\leqslant N_0} \|\partial_1u\|_{H^{0,s_0}} 2^{-k's}c_{k'} \|w\|_{H^{\frac{1}{4},s}}2^{-ks} c_k \|w\|_{H^{\frac{1}{4},s}}\\
\leqslant&Cc_k c^{(1)}_k2^{-2ks}\|\partial_1u\|_{H^{0,s_0}}\|w\|_{H^{\frac{1}{4},s}}\|w\|_{H^{\frac{1}{4}.s}},
\end{align*}
where $c^{(1)}_k=2^{ks}\sum_{|k-k'|\leqslant N_0} 2^{-k's}c_{k'}\in l^2$

Combining all the term together, let 
$$d_k'=c_k(b^{(1)}_k+a^{(1)}_k+a^{(2)}_k+\tilde{c}^{(2)}_k +c_k+c^{(1)}_k)\in l^1\quad \text{and } d_k=\frac{d_k'}{\|d_k'\|_{l^1}}$$
then the result holds.

\end{proof}

\bibliography{Ref0}{}  
\bibliographystyle{alpha}

\end{document}